\documentclass[12pt]{amsart}
\usepackage{}
\usepackage{amsmath}
\usepackage{amsfonts}
\usepackage{amssymb}
\usepackage[all,cmtip]{xy}           
\usepackage{bbm}
\usepackage{bbding}
\usepackage{txfonts}
\usepackage[shortlabels]{enumitem}
\usepackage{ifpdf}
\ifpdf
\usepackage[colorlinks,final,backref=page,hyperindex]{hyperref}
\else
\usepackage[colorlinks,final,backref=page,hyperindex,hypertex]{hyperref}
\fi
\usepackage{amscd}
\usepackage{tikz-cd}
\usepackage[active]{srcltx}

\makeatletter

\newtheorem{theorem}{Theorem}[section]
\newtheorem{prop}[theorem]{Proposition}
\newtheorem{lemma}[theorem]{Lemma}
\newtheorem{coro}[theorem]{Corollary}
\newtheorem{prop-def}{Proposition-Definition}[section]

\theoremstyle{definition}
\newtheorem{defn}[theorem]{Definition}

\newtheorem{remark}[theorem]{Remark}
\newtheorem{exam}[theorem]{Example}
\newtheorem{Ques}[theorem]{Question}

\newtheorem{List}[theorem]{List}

\newcommand{\nc}{\newcommand}

\nc{\delete}[1]{{}}
\nc{\mmargin}[1]{}
\nc{\mlabel}[1]{\label{#1}}  
\nc{\mcite}[1]{\cite{#1}}  
\nc{\mref}[1]{\ref{#1}}  
\nc{\mbibitem}[1]{\bibitem{#1}} 

\delete{
	\nc{\mlabel}[1]{\label{#1}  
		{\hfill \hspace{1cm}{\bf{{\ }\hfill(#1)}}}}
	\nc{\mcite}[1]{\cite{#1}{{\bf{{\ }(#1)}}}}  
	\nc{\mref}[1]{\ref{#1}{{\bf{{\ }(#1)}}}}  
	\nc{\mbibitem}[1]{\bibitem[\bf #1]{#1}} 
}

 \font\cyrs=wncyr7

\newcommand{\bk}{{\mathbf{k}}}

\nc{\vep}{\varepsilon}
\nc{\bin}[2]{ (_{\stackrel{\scs{#1}}{\scs{#2}}})}  
\nc{\binc}[2]{(\!\! \begin{array}{c} \scs{#1}\\
		\scs{#2} \end{array}\!\!)}  
\nc{\bincc}[2]{  ( {\scs{#1} \atop
		\vspace{-1cm}\scs{#2}} )}  
\nc{\oline}[1]{\overline{#1}}
\nc{\mapm}[1]{\lfloor\!|{#1}|\!\rfloor}
\nc{\bs}{\bar{S}}
\nc{\la}{\longrightarrow}
\nc{\ot}{\otimes}
\nc{\rar}{\rightarrow}
\nc{\lon }{\,\rightarrow\,}
\nc{\dar}{\downarrow}
\nc{\dap}[1]{\downarrow \rlap{$\scriptstyle{#1}$}}
\nc{\defeq}{\stackrel{\rm def}{=}}
\nc{\dis}[1]{\displaystyle{#1}}
\nc{\dotcup}{\ \displaystyle{\bigcup^\bullet}\ }
\nc{\hcm}{\ \hat{,}\ }
\nc{\hts}{\hat{\otimes}}
\nc{\hcirc}{\hat{\circ}}
\nc{\lleft}{[}
\nc{\lright}{]}
\nc{\curlyl}{\left \{ \begin{array}{c} {} \\ {} \end{array}
	\right .  \!\!\!\!\!\!\!}
\nc{\curlyr}{ \!\!\!\!\!\!\!
	\left . \begin{array}{c} {} \\ {} \end{array}
	\right \} }
\nc{\longmid}{\left | \begin{array}{c} {} \\ {} \end{array}
	\right . \!\!\!\!\!\!\!}
\nc{\ora}[1]{\stackrel{#1}{\rar}}
\nc{\ola}[1]{\stackrel{#1}{\la}}
\nc{\scs}[1]{\scriptstyle{#1}} \nc{\mrm}[1]{{\rm #1}}
\nc{\dirlim}{\displaystyle{\lim_{\longrightarrow}}\,}
\nc{\invlim}{\displaystyle{\lim_{\longleftarrow}}\,}
\nc{\dislim}[1]{\displaystyle{\lim_{#1}}} \nc{\colim}{\mrm{colim}}
\nc{\mvp}{\vspace{0.3cm}} \nc{\tk}{^{(k)}} \nc{\tp}{^\prime}
\nc{\ttp}{^{\prime\prime}} \nc{\svp}{\vspace{2cm}}
\nc{\vp}{\vspace{8cm}}
\nc{\modg}[1]{\!<\!\!{#1}\!\!>}
\nc{\intg}[1]{F_C(#1)}
\nc{\lmodg}{\!<\!\!}
\nc{\rmodg}{\!\!>\!}
\nc{\cpi}{\widehat{\Pi}}
\nc{\ssha}{{\mbox{\cyrs X}}} 
\nc{\tsha}{{\mbox{\cyrt X}}}
\nc{\shpr}{\diamond}    
\nc{\labs}{\mid\!}
\nc{\rabs}{\!\mid}
 \nc{\zhx}{\text{-}}

\nc{\ad}{\mrm{ad}}
\nc{\ann}{\mrm{ann}}
\nc{\Aut}{\mrm{Aut}}
\nc{\Av}{\mrm{Av}}
\nc{\bim}{\mbox{-}\mathsf{Bimod}}
\nc{\br}{\mrm{bre}}
\nc{\can}{\mrm{can}}
\nc{\Cont}{\mrm{Cont}}
\nc{\rchar}{\mrm{char}}
\nc{\cok}{\mrm{coker}}
\nc{\de}{\mrm{dep}}
\nc{\Dif}{\mrm{Diff}}
\nc{\dtf}{{R-{\rm tf}}}
\nc{\dtor}{{R-{\rm tor}}}

\nc{\Div}{{\mrm Div}}
\nc{\Diff}{\mrm{DA}}
\nc{\Diffl}{\mathsf{DA}_\lambda}
\nc{\diffo}{{\mathsf{DO}_\lambda}}
\nc{\alg}{\mathsf{Alg}}
\nc{\End}{\mrm{End}}
\nc{\Ext}{\mrm{Ext}}
\nc{\Fil}{\mrm{Fil}}
\nc{\Fr}{\mrm{Fr}}
\nc{\Frob}{\mrm{Frob}}
\nc{\Gal}{\mrm{Gal}}
\nc{\GL}{\mrm{GL}}
\nc{\Hom}{\mrm{Hom}}
\nc{\Hoch}{\mrm{Hoch}}
\nc{\hsr}{\mrm{H}}
\nc{\hpol}{\mrm{HP}}
\nc{\id}{\mrm{id}}
\nc{\im}{\mrm{im}}
\nc{\Id}{\mrm{Id}}
\nc{\ID}{\mrm{ID}}
\nc{\Irr}{\mrm{Irr}}
\nc{\incl}{\mrm{incl}}
\nc{\length}{\mrm{length}}
\nc{\NLSW}{\mrm{NLSW}}
\nc{\Lie}{\mrm{Lie}}
\nc{\Nij}{\mrm{Nij}}
\nc{\mchar}{\rm char}
\nc{\mpart}{\mrm{part}}
\nc{\ql}{{\QQ_\ell}}
\nc{\qp}{{\QQ_p}}
\nc{\rank}{\mrm{rank}}
\nc{\rcot}{\mrm{cot}}
\nc{\rdef}{\mrm{def}}
\nc{\rdiv}{{\rm div}}
\nc{\Rey}{\mrm{Rey}}
\nc{\rtf}{{\rm tf}}
\nc{\rtor}{{\rm tor}}
\nc{\res}{\mrm{res}}
\nc{\SL}{\mrm{SL}}
\nc{\Spec}{\mrm{Spec}}
\nc{\tor}{\mrm{tor}}
\nc{\Tr}{\mrm{Tr}}
\nc{\tr}{\mrm{tr}}
\nc{\wt}{\mrm{wt}}
\def\ot{\otimes}

\nc{\bfk}{{\bf k}}
\nc{\bfone}{{\bf 1}}
\nc{\bfzero}{{\bf 0}}
\nc{\detail}{\marginpar{\bf More detail}
	\noindent{\bf Need more detail!}
	\svp}
\nc{\gap}{\marginpar{\bf Incomplete}\noindent{\bf Incomplete!!}
	\svp}
\nc{\FMod}{\mathbf{FMod}}
\nc{\Int}{\mathbf{Int}}
\nc{\remarks}{\noindent{\bf Remarks: }}
\nc{\ob}{\mathsf{Ob}}

\nc{\BA}{{\mathbb A}}   \nc{\CC}{{\mathbb C}}
\nc{\DD}{{\mathbb D}}   \nc{\EE}{{\mathbb E}}
\nc{\FF}{{\mathbb F}}   \nc{\GG}{{\mathbb G}}
\nc{\HH}{{\mathbb H}}   \nc{\LL}{{\mathbb L}}
\nc{\NN}{{\mathbb N}}   \nc{\PP}{{\mathbb P}}
\nc{\QQ}{{\mathbb Q}}   \nc{\RR}{{\mathbb R}}
\nc{\TT}{{\mathbb T}}   \nc{\VV}{{\mathbb V}}
\nc{\ZZ}{{\mathbb Z}}   \nc{\TP}{\widetilde{P}}

\nc{\cala}{{\mathcal A}}    \nc{\calc}{{\mathcal C}}
\nc{\cald}{\mathcal{D}}     \nc{\cale}{{\mathcal E}}
\nc{\calf}{{\mathcal F}}    \nc{\calg}{{\mathcal G}}
\nc{\calh}{{\mathcal H}}    \nc{\cali}{{\mathcal I}}
\nc{\call}{{\mathcal L}}    \nc{\calm}{{\mathcal M}}
\nc{\caln}{{\mathcal N}}    \nc{\calo}{{\mathcal O}}
\nc{\calp}{{\mathcal P}}    \nc{\calr}{{\mathcal R}}
\nc{\cals}{{\mathcal S}}    \nc{\calt}{{\Omega}}
\nc{\calv}{{\mathcal V}}    \nc{\calw}{{\mathcal W}}
\nc{\calx}{{\mathcal X}}    \nc{\calu}{{\mathcal U}}
\nc{\caly}{{\mathcal Y}}

\nc{\uOpAlg}{\mathfrak{uOpAlg}}

\nc{\OpAlg}{\mathfrak{OpAlg}}

\nc{\ComOpAlg}{\mathfrak{ComOpAlg}}

\nc{\OpVect}{\mathfrak{OpVect}}

\nc{\OpSet}{\mathfrak{OpSet}}

\nc{\OpMon}{\mathfrak{OpMon}}

\nc{\ComOpMon}{\mathfrak{ComOpMon}}

\nc{\OpSem}{\mathfrak{OpSem}}

\nc{\ComOpSem}{\mathfrak{ComOpSem}}

\nc{\uAlg}{\mathfrak{uAlg}}

\nc{\Alg}{\mathfrak{Alg}}

\nc{\ComAlg}{\mathfrak{ComAlg}}

\nc{\Vect}{\mathfrak{Vect}}

\nc{\Set}{\mathfrak{Set}}

\nc{\Mon}{\mathfrak{Mon}}

\nc{\ComMon}{\mathfrak{ComMon}}

\nc{\Sem}{\mathfrak{Sem}}

\nc{\ComSem}{\mathfrak{ComSem}}
\nc{\fraka}{{\mathfrak a}}
\nc{\frakb}{\mathfrak{b}}
\nc{\frakg}{{\frak g}}
\nc{\frakl}{{\frak l}}
\nc{\fraks}{{\frak s}}
\nc{\frakB}{{\frak B}}
\nc{\frakm}{{\frak m}}
\nc{\frakM}{{\mathfrak M}}
\nc{\frakp}{{\frak p}}
\nc{\frakW}{{\frak W}}
\nc{\frakX}{{\frak X}}
\nc{\frakS}{{\frak S}}
\nc{\frakA}{{\frak A}}
\nc{\frakx}{{\frakx}}

\nc{\lir}[1]{\textcolor{red}{\underline{Li:}#1 }}

\begin{document}

\title[Free objects]{Free objects and Gr\"obner-Shirshov bases  in operated contexts}

\author{Zihao Qi, Yufei Qin, Kai Wang and Guodong Zhou}
\address{  School of Mathematical Sciences, Shanghai Key Laboratory of PMMP,
  East China Normal University,
 Shanghai 200241,
   China}
   \email{qizihao@foxmail.com}
   \email{290673049@qq.com}
\email{wangkai@math.ecnu.edu.cn }
 \email{gdzhou@math.ecnu.edu.cn}

\date{\today}

\begin{abstract} This paper investigates algebraic objects equipped with an operator, such as operated monoids, operated algebras etc.  Various free  object functors   in these operated contexts are explicitly constructed. For operated algebras whose   operator satisfies  a set $\Phi$ of  relations
 (usually called operated polynomial identities (aka. OPIs)),  Guo defined free objects, called free $\Phi$-algebras,      via universal algebra.
 Free $\Phi$-algebras over algebras are studied in details.
 A mild sufficient condition is found
   such that $\Phi$  together with a    Gr\"obner-Shirshov basis of an algebra $A$    form a Gr\"obner-Shirshov basis of  the free $\Phi$-algebra over  algebra $A$   in the sense of Guo et al..  Ample examples for which this condition holds are provided, such as      all Rota-Baxter type OPIs, a class of differential type OPIs, averaging OPIs and Reynolds OPI.

\end{abstract}

\subjclass[2010]{
13P10 
03C05 
08B20 
12H05 
16S10 
}

\keywords{Differential type OPIs; free objects; Gr\"obner-Shirshov bases; operated algebras; free operated algebras over algebras; operated polynomial identities;    Rota-Baxter type OPIs; universal algebra}

\maketitle

 \tableofcontents

\allowdisplaybreaks

\section*{Introduction}

This paper studies explicit constructions of free objects and Gr\"obner-Shirshov bases  in operated contexts.

\smallskip

 In the 1930s, Ritt \cite{Rit,Rit1}  initiated the algebraic study of   differential equations by introducing differential algebras as  commutative algebras equipped with a linear operator satisfying the usual Leibniz rule.
  This mathematical branch has received ample development in the work    \cite{Kol,Mag,PS} and has broad applications to other areas such as
    arithmetic geometry, logic,  computer science  and mathematical physics \cite{CM,FLS,MS,Wu,Wu2} etc.
In recent years, researchers began to investigate noncommutative differential algebras  in order to broaden the scope of the theory to  include   path algebras, for instance,  and to have a more meaningful differential Lie algebra theory  \cite{GL, Poi, Poi1} and also from  an operadic point of view \cite{Loday}

\smallskip

Another important class of algebras with operators are Rota-Baxter algebras.
These algebras (previously known as  Baxter algebras) originated with the work of    Baxter \cite{Bax}   on probability theory.
Baxter's work was further investigated by,  among others,   Rota \cite{Rota69}  (hence the name  ``Rota-Baxter algebras''),   Cartier \cite{Cartier} and   Atkinson \cite{Atkinson} etc.
Nowadays, Rota-Baxter algebras have  numerous  applications and connections to many  mathematical branches,  to name a few, such as combinatorics \cite{GuGuo,Rota95},  renormalization
in quantum field theory   \cite{CK}, multiple zeta values in number theory \cite{Guozhang},   operad theory \cite{Aguiar,BBGN},
  Hopf algebras \cite{CK}, Yang-Baxter equation \cite{Bai} etc. For basic theory about Rota-Baxter algebras, we refer the reader to the short introduction \cite{Guo09b} and to  the comprehensive monograph \cite{Guo12}.

 \medskip

Free objects are ubiquitous in mathematics. The general idea is to present an object by realizing it as a quotient of a free object, whence    presentations by generators and relations. Explicit construction of free objects are usually very important in many subjects.
Earlier construction of commutative  free Rota-Baxter algebras were given by Rota \cite{Rota69}  and Cartier \cite{Cartier}. In the pioneering work \cite{GuoKeigher00a,GuoKeigher00b}, Guo and Keigher   gave the   third construction   in terms of   mixable shuffle product algebras, generalizing the well-known construction of shuffle
product algebras.
 Several groups of authors    gave explicit constructions of free noncommutative
Rota-Baxter algebras on sets, modules or algebras \cite{AM,ERG08a,ERG08b,Guo09a,GuoSit}.

Guo  \cite{Guo09a} revisited these constructions from the viewpoint of semigroups, monoids or algebras endowed with operators in the sense of Higgins \cite{Higgins} and Kurosh \cite{Kurosh}, which Guo renamed as  operated semigroups, operated monoids or operated algebras.
He constructed the  free operated semigroup,  the free operated monoid  and the free operated algebras from a set in terms of bracketed words,  Motzkin paths and rooted trees; see also \cite{DH,BokutChenQiu,ZhangGao}.
When the operator in question satisfies certain relations,   Guo, Sit and Zhang   \cite{GSZ} introduced the notion of  operated polynomial identities (aka. OPIs). By the theory of  universal algebra,  for a set $\Phi$ of OPIs,  the free $\Phi$-algebra can be realised as the quotient of the free operated algebras by the operated ideal generated by $\Phi$; see, for instance, \cite[Proposition 1.3.6]{Cohn}.
This gives a universal construction for free $\Phi$-algebras.

\smallskip

 The next step is to develop
a   theory of rewriting systems and that of Gr\"{o}bner-Shirshov (aka. GS) bases  in this operated context, parallel to and include as a special case,  the usual well developed  theory  for associative algebras \cite{Shirshov,Buc,Green,BokutChen14,BokutChen20} and the original theory of rewriting systems \cite{BN}.   This  has been done in \cite{BokutChenQiu,GSZ,GGSZ,GaoGuo17} and was applied to various setting,  in particular, to two important classes of OPIs:  differential type OPIs and Rota-Baxter type OPIs which were carefully studied by Guo et al. \cite{GSZ}\cite{GGSZ}\cite{GaoGuo17}.  Recently, there is a need to develop free $\Phi$-algebras over algebras and construct   GS  bases for these free $\Phi$-algebras as long as a GS  basis is known for the given algebra; see \cite{ERG08b,LeiGuo,GuoLi}.

\medskip

This paper extends and refines some aspects of the above works.

\smallskip

In the first three sections,
we present a careful analysis for explicit construction of free objects in various operated contexts.

Denote the category of  sets (resp. semigroups, monoids)  by $\Set$ (resp. $\Sem$, $\Mon$).
We consider categories of structured  sets  endowed with an operator which is merely a map without any further condition, such as  operated sets, operated semigroups and  operated monoids; see Definition~\ref{Def: operated sets to monoids}.
Denote  the category of operated sets, that of  operated semigroups and that of operated monoids by  $ \OpSet $, $ \OpSem $ and $ \OpMon $,   respectively.
We also investigate vector spaces with structures endowed with a linear  operator such as  operated vector spaces, operated algebras and operated unital algebras; see   Definition~\ref{Def: operated vector spaces to unital algebras}. We use  $ \OpVect $ (resp. $ \OpAlg $, $  \uOpAlg $) to denote the category of operated $\bk$-vector spaces (resp. operated $\bk$-algebras,  operated unital $\bk$-algebras).

Our goal in the first three sections is to complete the following diagram of free object functors:
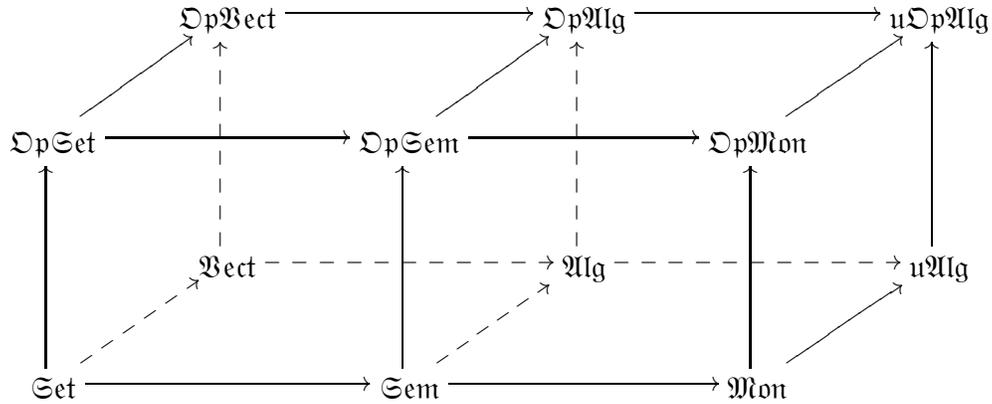
\[\xymatrixrowsep{2.5pc}
\xymatrix{
	& \OpVect \ar@<.5ex>@{->}[rr]
	& &\OpAlg   \ar@<.5ex>@{->}[rr]
	&  & \uOpAlg   \\
	\OpSet \ar@<.5ex>@{->}[ur] \ar@<.5ex>@{->}[rr]
	& &\OpSem  \ar@<.5ex>@{->}[ur]   \ar@<.5ex>@{->}[rr]
	& & \OpMon  \ar@<.5ex>@{->}[ur]    &\\
	&  \Vect  \ar@<.5ex>@{-->}[uu]  \ar@<.5ex>@{-->}[rr]
	&  &  \Alg   \ar@<.5ex>@{-->}[uu]  \ar@<.5ex>@{-->}[rr]
	& &  \uAlg  \ar@<.5ex>@{->}[uu]   \\
	\Set \ar@<.5ex>@{->}[uu] \ar@<.5ex>@{-->}[ur] \ar@<.5ex>@{->}[rr]
	&  & \Sem   \ar@<.5ex>@{->}[rr] \ar@<.5ex>@{->}[uu]\ar@<.5ex>@{-->}[ur]
	&  & \Mon  \ar@<.5ex>@{->}[ur]  \ar@<.5ex>@{->}[uu]  &   }
\]
Note that these free object functors  are left adjoint to    the obvious forgetful functors which  are not drawn in the diagram.

In the first section, we consider the bottom  face and part of the  upper  face,  the second section contains an analysis of the front  face and the back  face is dealt with in the third section.
  Except the   functors in the  bottom face, it seems that any other free object functor has not been  appeared in the literature, while some composites were pointed out by Guo \cite{Guo09a}.

\smallskip

The fourth section recalls the basic   theory of   operated contexts \cite{Guo09a,GSZ,GGSZ,GaoGuo17}. In particular, as an application of the previous three sections, we can construct, via universal algebra,   the free $\Phi$-algebra over  a given algebra for a set of OPIs $\Phi$;  see Propositions \ref{Prop: free Phi-algebra over  a given algebra via universal algebra} and \ref{Prop: from unital algebra to Phi algebra with relations}.

\smallskip

The first half of the fifth section contains an account of the    theory  of  GS  basis in operated contexts \cite{BokutChenQiu,DH,GSZ,GGSZ,GaoGuo17}.  In the second half of Section 5, inspired by \cite{ERG08b,LeiGuo,GuoLi},  we are interested into a question which can be roughly expressed as follows:
\begin{Ques} Given a unital  algebra  $A$ with a   GS basis  $G$ and  a set of OPIs $\Phi$,
	  assume that   $\Phi$ is GS  in the sense of \cite{BokutChenQiu,GSZ,GGSZ,GaoGuo17}.  Let $B$
be the free $\Phi$-algebra over $A$. Under what conditions, $\Phi\cup G$ will be  a GS  basis  for $    B$?
\end{Ques}
We answer this question in the affirmative under a mild condition in Theorem~\ref{Thm: GS basis for free Phi algebra over alg}, which can be considered as the main result of this paper. When this   condition is satisfied, $\Phi\cup G$  is  a GS  basis  for $    B$.  As a consequence, we also get  a linear basis of $B$; see Corollary~\ref{Cor: linear basis for free Phi algebra over alg}.

\smallskip

Section 6 contains  many applications of Theorem~\ref{Thm: GS basis for free Phi algebra over alg} and Corollary~\ref{Cor: linear basis for free Phi algebra over alg}. When $\Phi$ consists of a single OPI of Rota-Baxter type in the sense of \cite{GGSZ, GaoGuo17}, this technical condition is fulfilled; see Theorem~\ref{Thm: GS basis for free Rota type over algebra}. When $\Phi$ consists of a single OPI of differential type in the sense of \cite{GSZ, GaoGuo17}, the situation is much more involved and we give two particular cases where this technical condition holds; see Proposition~\ref{Prop: GS basis for differential type 1+3}.
  We also provide a case study for Gr\"obner-Shirshov $\Phi$ consisting of several OPIs. For averaging algebras,  $\Phi$ has three elements and for Reynolds algebras, $\Phi$ will have  infinitely many  elements;
based on \cite{GaoZhang, ZhangGaoGuo}, we could obtain a    GS  basis as well as  a linear basis for free averaging algebras over algebras and free Reynolds algebras over algebras, respectively.

\medskip

\textbf{Notation:} Throughout this paper, $\bk$ denotes a field. All the vector spaces and algebras are over $\bk$ and all tensor products are also taking over $\bk$.

\medskip

 \textbf{Acknowledgements:}   The authors were  supported    by NSFC (No. 11671139, 11971460, 12071137)   and  by  STCSM  (No. 18dz2271000).

  The authors are grateful to Profs.  Xing Gao,  Li Guo  and  Yunnan Li      for their inspiring online talks and many useful comments.

\section{Free object functors I}

Firstly, we introduce some relevant definitions and notations.

Recall that $\Set$ (resp. $\Sem$, $\Mon$, $\Vect$, $\Alg$, $\uAlg$)  denotes the category of  sets (resp. semigroups, monoids, $\bk$-vector spaces,   $\bk$-algebras, unital $\bk$-algebras).

The following two defintions are taken from \cite{Guo09a,GGSZ} except the notion of operated sets.
\begin{defn} \label{Def: operated sets to monoids}
	\begin{itemize}
    \item[(a)] An operated set is a set $X$  with a map $P_X: X \rightarrow X$.
	 A morphism of operated sets from $(X, P_X)$ to $(Y, P_Y)$ is a map $ f: X \rightarrow Y$ such that $f\circ P_X = P_Y\circ f $. 
 Denote the category of operated sets by $ \OpSet $.
	
	\item[(b)] An operated semigroup is a semigroup $S$ with a map $P_S:S\rightarrow S$ (which is not necessarily a homomorphism of semigroups). Morphisms between operated semigroups can be defined in the obvious way.  Denote the category of operated semigroups by $ \OpSem $.
	
	\item[(c)]  An operated monoid  is a monoid
$ M $ together with a map $ P_{M}:M\rightarrow M $ (which is not necessarily a homomorphism of monoids).  One can define morphisms between operated monoids   similarly.  Denote the category of operated monoids by $ \OpMon $.
\end{itemize}

\end{defn}

\begin{defn} \label{Def: operated vector spaces to unital algebras}
	An operated $\bk$-space (resp. operated $\bk$-algebra, operated unital $\bk$-algebra) is a  $\bk$-space (resp. $\bk$-algebra, unital $\bk$-algebra)	  $V$  together with a $\bk$-linear map $ P_{V}: V\rightarrow V $. A morphism between   operated $\bk$-vector spaces (resp. operated $\bk$-algebras, operated unital $\bk$-algebras) $ (V, P_V) $ and  $(W, P_W)  $ is a homomorphism of  $\bk$-vector spaces (resp. $\bk$-algebras, unital $\bk$-algebras) $ f : V \rightarrow W $ such that $f\circ P_V = P_W\circ f $.


	We use $ \OpVect $ (resp. $ \OpAlg $,  $ \uOpAlg $) to denote the category of operated $\bk$-vector spaces (resp. operated $\bk$-algebras, operated unital $\bk$-algebras ).
\end{defn}

Now, we are going to complete  the bottom face and  part of the upper face of the cubic diagram displayed in Introduction.

\subsection{The bottom face}\

Firstly, let's consider    the following diagram:
$$\xymatrixcolsep{6pc}\xymatrixrowsep{5.5pc}
\xymatrix{ \Vect \ar@<1ex>[r]^{\calf^{\Alg}_{\Vect}}    \ar@<1ex>[d]^{\calu^{\Vect}_{\Set}}
	& \Alg  \ar@<1ex>[d]^{\calu^{\Alg}_{\Sem}}  \ar@<1ex>[l]^{\calu_{\Vect}^{\Alg}}  \ar@<1ex>[r]^{\calf^{\uAlg}_{\Alg}}
	& \uAlg \ar@<1ex>[l]^{\calu_{\Alg}^{\uAlg}}   \ar@<1ex>[d]^{\calu^{\uAlg}_{\Mon}} \\
	\Set  \ar@<1ex>[r]^{\calf^{\Sem}_{\Set}}      \ar@<1ex>[u]^{\calf_{\Set}^{\Vect}}
	&  \Sem   \ar@<1ex>[l]^{\calu_{\Set}^{\Sem}}\ar@<1ex>[r]^{\calf^{\Mon}_{\Sem}}  \ar@<1ex>[u]^{\calf_{\Sem}^{\Alg}}
	&\Mon \ar@<1ex>[u]^{\calf_{\Mon}^{\uAlg}}  \ar@<1ex>[l]^{\calu_{\Sem}^{\Mon}}}$$
where,  as explained in Introduction, the    symbols  $\calu^*_*$ stand for  forgetful functors and
the free object functors    $\calf^*_*$  are left adjoint to the corresponding forgetful functors.

It is well known that the functor  $\calf^{\Sem}_{\Set}$ assigns to a set $X$ the free semigroup $\cals(X):=\calf^{\Sem}_{\Set}(X)$  generated by it; the functor $\calf_{\Sem}^{\Mon}$ is given by adding a unit element to a semigroup; as a consequence, the composite $\calf_{\Set}^{\Mon}:=\calf_{\Sem}^{\Mon}\circ \calf_{\Set}^{\Sem}$ is just  the usual free monoid functor sending  a set $X$ to the free monoid $\calm(X):=\calf_{\Set}^{\Mon}(X)$ generated by $X$.

For any $\bk$-space $V$, $\calf_{\Vect}^{\Alg}(V)$ is the reduced tensor algebra $\bar{\TT}(V)=\bigoplus_{n\geqslant 1}V^{\ot n}$;  given a non-unital algebra $A$, the algebra $\calf_{\Alg}^{\uAlg}(A)$ is defined to be the  space $\bk\oplus A$ endowed with multiplication $$(\lambda, a)(\mu, b)=(\lambda \mu, \lambda b+\mu a+a  b)$$ for $\lambda, \mu \in \bk$ and $a, b\in A$, which is a unital $\bk$-algebra with unit $(1_\bk, 0)$; for any $\bk$-space $V$, the free unital $\bk$-algebra  $\calf_{\Vect}^{\uAlg}(V)=\calf_{\Alg}^{\uAlg} \circ \calf_{\Vect}^{\Alg}(V)$  generated by $V$ is just the tensor algebra  $\TT(V)=\bk\oplus \bar{\TT}(V)=\bigoplus_{n\geqslant 0}V^{\ot n}$. %

 While the vertical forgetful functors $\calu_*^*$ in the above diagram  is   ignoring the $\bk$-space structure, the vertical free object functors which are left adjoint to the corresponding forgetful functors are exactly given by  linearization. For a set (resp.  semigroup, monoid) $X$,  $\calf_{\Set}^{\Vect}(X)$ (resp. $\calf_{\Sem}^\Alg(X)$, $\calf_{\Mon}^\uAlg(X)$)  is exactly $\bk X$, the $\bk$-space spanned by elements in $X$, which is naturally a vector space (resp. algebra, unital algebra).    Denote $\calf_{\Set}^{\uAlg}=\calf_{\Alg}^{\uAlg}\circ \calf_{\Vect}^{\Alg}\circ\calf_{\Set}^{\Vect}$.  Then for a set $X$, $\calf_{\Set}^{\uAlg}(X)$ is the free unital algebra  generated by $X$, which is just $\bk \calm(X)$, the  noncommutative polynomial algebra with variables in $X$.

\subsection{The upper face}
Now, let's clarify the vertical arrows of the following diagram:
 $$\xymatrixcolsep{5pc}\xymatrixrowsep{5pc}
\xymatrix{ \OpVect \ar@<1ex>[r]    \ar@<1ex>[d]^{\calu^{\OpVect}_{\OpSet}}
	& \OpAlg  \ar@<1ex>[d]^{\calu^{\OpAlg}_{\OpSem}}  \ar@<1ex>[l]
\ar@<1ex>[r]
	& \uOpAlg \ar@<1ex>[l]   \ar@<1ex>[d]^{\calu^{\uOpAlg}_{\OpMon}} \\
 \OpSet  \ar@<1ex>[r]      \ar@<1ex>[u]^{\calf_{\OpSet}^{\OpVect}}
	&  \OpSem   \ar@<1ex>[l]\ar@<1ex>[r]  \ar@<1ex>[u]^{\calf_{\OpSem}^{\OpAlg}}
	&\OpMon \ar@<1ex>[u]^{\calf_{\OpMon}^{\uOpAlg}}  \ar@<1ex>[l]}$$

Same as in the previous subsection, the vertical forgetful functors are just forgetting the $\bk$-space structure  and the vertical free object functors are also given by linearization.
 For an operated set (resp. operated semigroup, operated monoid)  $(X, P_X)$, the associated operated vector space  $\calf_{\OpSet}^{\OpVect}(X,P_X)$ (resp. operated algebra $\calf_{\OpSem}^{\OpAlg}(X,P_X)$, unital operated algebra $\calf_{\OpMon}^{\uOpAlg}(X,P_X)$) is just $(\bk X, P_{\bk X})$, where the linear operator $P_{\bk X}$ on $\bk X$ is just the linear extension of the operator $P_X$ on $X$.

\section{Free object functors II}
In this section,  we want  to investigate  the front face of the   diagram displayed in Introduction.  More precisely,  we will construct  all the free object functors in  the following diagram of functors:
   $$\xymatrixcolsep{5pc}\xymatrixrowsep{5pc}
 \xymatrix{\OpSet\ar@<1ex>[r]^{\calf^{\OpSem}_{\OpSet}}   \ar@<1ex>[d]^{\calu^{\OpSet}_{\Set}}
 & \OpSem\ar@<1ex>[r]^{\calf^{\OpMon}_{\OpSem}}\ar@<1ex>[l]^{\calu_{\OpSet}^{\OpSem}} \ar@<1ex>[d]^{\calu^{\OpSem}_{\Sem}}
  &  \OpMon\ar@<1ex>[l]^{\calu_{\OpSem}^{\OpMon}}\ar@<1ex>[d]^{\calu^{\OpMon}_{\Mon}}\\
 \Set\ar@<1ex>[u]^{\calf_{\Set}^{\OpSet}} \ar@<1ex>[r]^{\calf^{\Sem}_{\Set}}
 & \Sem\ar@<1ex>[r]^{\calf^{\Mon}_{\Sem}}\ar@<1ex>[u]^{\calf_{\Sem}^{\OpSem}}\ar@<1ex>[l]^{\calu_{\Set}^{\Sem}}
 &   \Mon \ar@<1ex>[u]^{\calf_{\Mon}^{\OpMon}}  \ar@<1ex>[l]^{\calu_{\Sem}^{\Mon}}}$$
Same as before,  the    symbols  $\calu^*_*$ represent   forgetful functors and
 the free object functors    $\calf^*_*$  are left adjoint to the corresponding forgetful functors.

\subsection{Free object functor from  sets to operated sets}\

Let $X$ be a nonempty set.  Denote by  $\lfloor X \rfloor$     the set of the formal elements $ \lfloor x \rfloor, x\in X$.
Write $\lfloor X\rfloor^{(0)}=X$,   $\lfloor X \rfloor^{(1)}=\lfloor X \rfloor$ and  for $n\geq 1$, put  $\lfloor X\rfloor^{(n+1)}=\lfloor \lfloor X\rfloor^{(n)}\rfloor$.  For $x\in X$, let $\lfloor x\rfloor^{(n)}$ be the corresponding element in $\lfloor X\rfloor^{(n)}$.
Define  $\calx$ to be the disjoint union of all  $\left \lfloor X\right \rfloor^{(n)}, n\geq 0$ and an operator $P_{\calx}$ on $\calx$ mapping $u\in \calx$ to $\left \lfloor u\right \rfloor$. Then  the pair $\calf^{\OpSet}_{\Set}(X):=(\calx, P_\calx)$ forms an operated set.  When $X$ is the empty set $\emptyset$, define $\calf^{\OpSet}_{\Set}(\emptyset)$ to be the empty set together with its identity map.

 For a map $ f: X\to Y$, introduce $\calf^{\OpSet}_{\Set}(f): (\calx, P_\calx) \to (\caly, P_\caly)$ to be the map sending $\lfloor x\rfloor^{(n)}$ with $x\in X$ to $\lfloor f(x)\rfloor^{(n)}$. It is easy to see that $\calf^{\OpSet}_{\Set}$ is a functor from $\Set $ to $\OpSet $.

\begin{prop}\label{Prop: Free object functor from sets to operated sets}
	The   functor $\calf^{\OpSet}_{\Set}:\Set \rightarrow \OpSet $ is left adjoint to the forgetful functor $\calu_{\Set}^{\OpSet}:\OpSet\rightarrow \Set$, hence giving the free operated set generated by a set.
\end{prop}

  \begin{proof} Let $X$ be a set and $(Y, P_Y)$ be an operated set. Given a map $\theta: X\to Y$, we will show that there is a unique morphism of operated sets $\tilde{\theta}: \calf^{\OpSet}_{\Set}(X) \to Y$ which, when restricted to $X$, recovers $\theta$.

	In fact,  it suffices to    define the  map $\tilde{\theta}$ from $\calx$ to $Y$ by imposing  $\tilde{\theta}(  \lfloor x\rfloor^{(n)}):=P_Y^n(\theta(x))$ for any $x\in X$. Clearly, $\tilde{\theta}$ is the unique morphism of operated sets from $(\calx, P_\calx)$ to $(Y,  P_Y)$ whose restriction on $X$ is equal to $\theta$.

Therefore, $\calf^{\OpSet}_{\Set}$ is the free object functor from $\Set$ to $\OpSet$.
	\end{proof}


\subsection{Free object functor    from semigroups to operated semigroups}\

In this subsection, we will construct the  functor  $ \calf^{\OpSem}_{\Sem}:\Sem\rightarrow \OpSem $ left adjoint to the forgetful functor $ \calu^{\OpSem}_{\Sem}:\OpSem\rightarrow \Sem $.

\begin{defn}
	Let $ Z $ be a set.   Recall for a semigroup $T$,   $\calu^{\Sem}_\Set(T)$ is just the underlying set of $T$.  Write $\cals(T\sqcup Z):=\cals( \calu^{\Sem}_\Set(T)\sqcup Z)$   the free semigroup generated by the disjoint union $\calu^{\Sem}_{\Set}(T)\sqcup Z$.
	We define a  semigroup 	$\tilde{\cals}(T\sqcup Z)$ to be  the quotient of $\cals(T\sqcup Z)$ by identifying, for all   $t,t'\in T$, their      products      in $T$ and   in $\cals(T\sqcup Z)$.

\end{defn}
Notice that  if $ Z=\emptyset $,   $\tilde{\cals}(T\sqcup Z)$ is exactly   $T$ unchanged.

The following lemma will be useful in the sequel, whose easy proof is left to the reader.

\begin{lemma}\label{Lem: unique expression in Stilde} \begin{itemize}
		\item[(a)]
		Let $ Z $ be a set and $ T $ be a semigroup. Each element in $\tilde{\cals}(T\sqcup Z)$ has a unique expression
		$$a_0t_1a_1t_2\cdots t_na_n$$
		with
$n\geq 0$ and all  $t_i\in T$, where  for any $0\leq i\leq n$, $a_i$ lies in $\cals(Z)$ or is the empty word  $\varnothing$,  whereas when $n\geq 2$, neither of $a_1, \cdots, a_{n-1}$ can  be $\varnothing$ and when $n=0$,   $a_0$ is not   empty either.
		
		\item[(b)] Let $ Z\hookrightarrow Z' $ be an injective map of  sets and  $ T\hookrightarrow T' $ be an injective homomorphism of  semigroups.	Then the induced  homomorphism of semigroups
		$$\tilde{\cals}(T\sqcup Z)\to  \tilde{\cals}(T'\sqcup Z')$$
		is injective as well.
	\end{itemize}
\end{lemma}
\medskip

Let $T$ be a semigroup.  Deduced from Lemma~\ref{Lem: unique expression in Stilde}, the inclusion into the first component  $ T\hookrightarrow T\sqcup  \left \lfloor T\right \rfloor$ induces an injective  semigroup homomorphism
$$ i_{0,1}: \calf_0(T):=T \hookrightarrow \calf_1(T):=\tilde{\cals}(T\sqcup  \left \lfloor T\right \rfloor).
$$

For $ n\geq2 $, assume  that we have constructed $\calf_{n-2}(T)$ and $\calf_{n-1}(T)$ such that $\calf_{n-1}(T)=\tilde{\cals}(T\sqcup  \left \lfloor \calf_{n-2}(T)\right \rfloor) $  endowed  with  an  injective homomorphism of semigroups
$$
	  i_{n-2,n-1}: \calf_{n-2}(T) \hookrightarrow \calf_{n-1}(T).
$$
  We define the semigroup
$$\calf_{n}(T):=\tilde{\cals}(T\sqcup  \left \lfloor \calf_{n-1}(T)\right \rfloor)
$$
 and  again by Lemma~\ref{Lem: unique expression in Stilde}, the natural  injection
$$
	\Id_T\sqcup \left \lfloor i_{n-2,n-1}\right \rfloor :T\sqcup\left \lfloor \calf_{n-2}(T)\right \rfloor\hookrightarrow T\sqcup\left \lfloor \calf_{n-1}(T)\right \rfloor
$$
induces an injective semigroup  homomorphism
$$	i_{n-1,n}: \calf_{n-1}(T)=\tilde{\cals}(T\sqcup  \left \lfloor \calf_{n-2}(T)\right \rfloor) \hookrightarrow  \calf_{n}(T) =\tilde{\cals}(T\sqcup  \left \lfloor \calf_{n-1}(T)\right \rfloor).$$
 Define $  \calf(T)=\varinjlim \calf_{n}(T) $  and the map  sending  $u\in \calf_n(T)$ to $\left \lfloor u\right \rfloor\in \calf_{n+1}(T)$ induces an  operator $P_{\calf(T)}$  on $\calf(T)$.  By the limit construction, there are injective semigroup homomorphisms    $ i_n:\calf_{n}(T)\hookrightarrow\calf(T),  n\geq 0$ and $i:T\hookrightarrow\calf(T)$. Thus we have constructed a functor:
$$\begin{aligned}
		\calf^{\OpSem}_{\Sem}: &\Sem&\longrightarrow &\ \OpSem\\
		&T&\longmapsto&\ ( \calf(T),P_{\calf(T)}).
	\end{aligned}$$

\begin{prop}\label{Prop: free object functor from semigroups to operated semigroups}
	The   functor $\calf^{\OpSem}_{\Sem}:\Sem \rightarrow \OpSem $ is left adjoint to the forgetful functor $\calu_{\Sem}^{\OpSem}:\OpSem\rightarrow \Sem$.
\end{prop}

  \begin{proof}  Let $T$ be a semigroup and  $(S, P_S)$ be an operated semigroup.
		Let $$\theta: T\to \calu_{\Sem}^{\OpSem}(S,P_S)= S$$ be a homomorphism of semigroups.
		We will construct a morphism of operated semigroups
		$$\tilde{\theta}: \calf^{\OpSem}_{\Sem}(T) \to (S,P_S)$$
		such that $\tilde{\theta}\circ i=\theta$. Moreover, it is unique.

		Let $\tilde{\theta}_0=\theta$.  Define a    map
		$\theta_1:T\sqcup  \left \lfloor T\right \rfloor\rightarrow S$  by sending  any element $t\in T$
		to $\theta(t)$ and $\left \lfloor t\right \rfloor$ to $ P_S(\theta(t))$.  The universal property of free semigroups gives a homomorphism of semigroups $\bar{\theta}_1: \cals(T\sqcup \left \lfloor T\right \rfloor) \to S$ extending $\theta_1$.  For any $t_1,t_2\in T$,  we have $$\bar{\theta}_{1}(t_1\cdot t_2)=\bar{\theta}_{1}(t_1)\cdot \bar{\theta}_{1}(t_2)=\theta(t_1)\cdot\theta( t_2)=\theta(t_1 \cdot  t_2)=\bar{\theta}_1(t_1 \cdot  t_2),$$ where  the first $\cdot$ is the multiplication in $\cals(T\sqcup \left \lfloor T\right \rfloor)$ and the last one in $T$, so there exists  an induced semigroup homomorphism $$\tilde{\theta}_{1}: \calf_{1}(T)=\tilde{\cals}(T\sqcup  \left \lfloor T\right \rfloor)\to S.$$

		Assume by induction that  we are given  homomorphisms of semigroups $\tilde{\theta}_k: \calf_{i}(T)\to S$ for $1\leq k\leq n$ such that
$\tilde{\theta}_k\circ i_{k-1, k}=\tilde{\theta}_{k-1}$.  Build a map
		$\theta_{n+1}: T\sqcup  \left \lfloor \calf_{n}(T)\right \rfloor\rightarrow S$  by sending any element $t\in T$
		to $\theta(t)$ and $\left \lfloor v\right \rfloor$ to $P_S(\tilde{\theta}_n(v))$ for arbitrary $v\in \calf_{n}(T)$.
		By the universal property of free semigroups, $ \theta_{n+1}$   induces   a homomorphism of semigroups
 $\bar{\theta}_{n+1}: \cals(T\sqcup  \left \lfloor \calf_{n}(T)\right \rfloor)\to S$. For any $t_1,t_2\in T$,  we have  $$\bar{\theta}_{n+1}(t_1\cdot t_2)=\bar{\theta}_{n+1}(t_1)\cdot \bar{\theta}_{n+1}(t_2)= \theta (t_1)\cdot \theta (t_2)= \theta (t_1\cdot t_2)=\bar{\theta}_{n+1}(t_1\cdot t_2),$$ where  the first $\cdot$ is the multiplication in $\cals(T\sqcup  \left \lfloor \calf_{n}(T)\right \rfloor$ and the last one in $T$, so there exists  an induced semigroup homomorphism
 $$\tilde{\theta}_{n+1}: \calf_{n+1}(T)=\tilde{\cals}(T\sqcup  \left \lfloor \calf_{n}(T)\right \rfloor)\to S.$$
		
		Hence,   we  have constructed  a series of semigroup homomorphisms $\tilde{\theta}_k: \calf_k(T)\rightarrow S, k\geqslant 0 $ which fit into the commutative diagram:
		$$
\xymatrixrowsep{3pc}\xymatrixcolsep{4pc}\xymatrix{T\ar[r]^{\theta}\ar@{=}[d] & S&  \\
		\calf_0(T)\ar@{-->}[ur]_<(.5){ \tilde{\theta}_0=\theta}\ar@{^{(}->}[r]^{i_{0,1}} &\calf_1(T)\ar@{-->}[u]_{\exists\tilde{\theta}_1}\ar@{^{(}->}[r]^{i_{1,2}}&\calf_2(T)\ar@{-->}[ul]_<(.2){ \exists \tilde{\theta_2}}\ar@{^{(}->}[r]^{i_{2,3}} &\cdots\ar@{^{(}->}[r]&\calf(T)\ar@{-->}[lllu]_{ \exists\tilde{\theta}}}\nonumber
$$
		The limit construction gives  a homomorphism of semigroups $ \tilde{\theta} $ from $ \calf(T) $ to $ S $, which
		by construction commutes with the operators, thus  $\tilde{\theta} $ is a morphism of operated semigroups.
		
		The uniqueness of $\tilde{\theta}$ is clear.
		
		We are done.


\end{proof}
\subsection{Free object functor from monoids to operated monoids}\

Recall that $\calm(X):=\calf^{\Mon}_{\Set}(X)$ is  the free monoid generated by set $X$.
\begin{defn}
	Let $ Z $ be a set and $ T $ be a semigroup (resp. monoid). Write $\calm(T\sqcup Z):=\calm( \calu^{\Sem}_\Set(T)\sqcup Z)$ (resp. $\calm( \calu^{\Mon}_{\Set}(T)\sqcup Z)$) for  the free monoid generated by the disjoint union $\calu^{\Sem}_{\Set}(T)\sqcup Z$ (resp. $\calu^{\Mon}_{\Set}(T)\sqcup Z$).
	
	When $T$ is a semigroup, we define a   monoid 	$\tilde{\calm}(T\sqcup Z)$ to be  the quotient of $\calm(T\sqcup Z)$ by identifying, for all   $t,t'\in T$, their      products      in $T$ and   in $\calm(T\sqcup Z)$;
	when $T$ is a monoid and $e$ is its unit,  we need to identify, furthermore,   $e$ with  the unit of    $\calm(T\sqcup Z)$.

\end{defn}

Notice that  if $ Z=\emptyset $,   $\tilde{\calm}(T\sqcup Z)$ is   $\calf^{\Mon}_{\Sem}(T)$ (i.e. just adding a unit to $T$) when $T$ is  a semigroup     and when $T$ is a monoid, it is still $T$ itself.

Similar to Lemma~\ref{Lem: unique expression in Stilde},  the obvious proof of the following lemma is left to the reader.
\begin{lemma}\label{Lem: unique expression in Mtilde} \begin{itemize}
		\item[(1)]
		Let $ Z $ be a set and $ T $ be a semigroup or a monoid. Any element in $\tilde{\calm}(T\sqcup Z)$ possesses  a unique expression
		$$a_0t_1a_1t_2\cdots t_na_n$$
		with $n\geq 0$, where  for $0\leq i\leq n$, $a_i\in \calm(Z)$, but  for each  $1\leq i\leq n-1$, $a_i\in \cals(Z)$; for $1\leq i\leq n$, $t_i$ belongs to  $T$,  but neither of them  is   the unit of $T$ when $T$ is a monoid.
		
		\item[(2)] Let $ Z\hookrightarrow Z' $  (resp. $ T\hookrightarrow T' $)   be an injective map  of  sets (resp. an injective homomorphism of  semigroups or   monoids).	Then the induced  homomorphism
		$$\tilde{\calm}(T\sqcup Z)\to  \tilde{\calm}(T'\sqcup Z')$$
		is injective, too.
	\end{itemize}
\end{lemma}
\medskip

Let $M$ be a monoid.   The free operated monoid $\calf_{\Mon}^{\OpMon}(M)$  generated by $M$ can be built   exactly as is done in the preceding subsection    except that   $\tilde{\cals}$ is replaced by $\tilde{\calm}$ defined above  and we can show the    functor $\calf_{\Mon}^{\OpMon}:\Mon \rightarrow \OpMon $ is left adjoint to the forgetful functor $\calu_{\Mon}^{\OpMon}:\OpMon\rightarrow \Mon$.

\subsection{Free object functor    from operated sets  to operated semigroups} \

In this subsection, we will construct the functor   $ \calf^{\OpSem}_{\OpSet}:\OpSet\rightarrow \OpSem $ left adjoint to  the forgetful functor $ \calu_{\OpSet}^{\OpSem}:\OpSem\rightarrow \OpSet $.  Its construction   is a little more involved than as in the preceding subsections.


  Let $ (X,P_X)$ be an operated set and $\calf_0(X):=\cals(X)$. Consider the set $  \cals(X)\backslash X $ obtained by deleting $X$ from $\cals(X)$  and let $ \calf_1(X):=\cals(X\sqcup  \left \lfloor \cals(X)\backslash X\right \rfloor).$
The natural inclusion $ X\hookrightarrow X\sqcup  \left \lfloor \cals(X)\backslash X\right \rfloor$, after  applying the free semigroup functor $\calf^{\Sem}_{\Set}$,   gives  a homomorphism of semigroups
$$ i_{0,1}: \calf_0(X)=\cals(X)  \hookrightarrow \calf_1(X)=\cals(X\sqcup  \left \lfloor \cals(X)\backslash X\right \rfloor).
$$

For $ n\geq 2 $,  suppose  given  a sequence of injective semigroup  homomorphisms
$$\calf_0(X)\overset{i_{0,1}}{\hookrightarrow}\calf_1(X)\overset{i_{1,2}}{\hookrightarrow} \cdots \overset{i_{n-3,n-2}}{\hookrightarrow} \calf_{n-2}(X)\overset{i_{n-2,n-1}}{\hookrightarrow}\calf_{n-1}(X)
$$ with $\calf_{k}(X) =\cals(X\sqcup  \left \lfloor \calf_{k-1}(X)\backslash X\right \rfloor)$ for all $1\leq k\leq n-1$,
    define as above  the semigroup
\begin{equation}\label{key}
\nonumber \calf_{n}(X):=\cals(X\sqcup  \left \lfloor \calf_{n-1}(X)\backslash X\right \rfloor).
\end{equation}
Deduced from the natural  injection
$$
\Id_X\sqcup \left \lfloor i_{n-2,n-1}|_{\calf_{n-2}(X)\backslash X}\right \rfloor:X\sqcup\left \lfloor \calf_{n-2}(X)\backslash X\right \rfloor\hookrightarrow X\sqcup\left \lfloor \calf_{n-1}(X)\backslash X\right \rfloor,$$
we obtain  an injective semigroup  homomorphism
\begin{equation}\label{key}
\nonumber i_{n-1,n}: \calf_{n-1}(X) =\cals(X\sqcup  \left \lfloor \calf_{n-2}(X)\backslash X\right \rfloor)\hookrightarrow \calf_{n}(X) =\cals(X\sqcup  \left \lfloor \calf_{n-1}(X)\backslash X\right \rfloor).
\end{equation}
Finally,  we define $$\calf(X):=\varinjlim \calf_{n}(X).$$ Obviously $\calf(X)$ is a semigroup.  By the limit construction, there are injective semigroup homomorphisms    $ i_n:\calf_{n}(X)\hookrightarrow\calf(X),  n\geq 0$ and $i:X\hookrightarrow\calf(X)$.
The  operator $P_{\calf(X)}$ over $\calf(X)$  can be constructed as follows:
$$\label{key}
    P_{\calf(X)}(u):=
\left\{
\begin{array}{lcl}
	i(P_X(x)) &\text{if} &u=i(x), \  \text{for some}\  x\in X, \\
	 i_{k+1}(\left\lfloor u'\right\rfloor) & \text{if} & u \notin i(X),\ \text{but}\ u=i_k(u'),\  \text{for some}\  u'\in \calf_k(X).
\end{array}
\right. $$
Then $(\calf(X),P_{\calf(X)})$ forms an operated semigroup. Now it is obvious that $i:(X, P_X)\to (\calf(X),  P_{\calf(X)})$ is a morphism of operated sets.

 Moreover, it is ready to see that the construction above is functorial. 

 We have constructed a functor:
\begin{equation}\label{key}
\nonumber
\begin{aligned}
	\calf^{\OpSem}_{\OpSet}: &\OpSet&\longrightarrow &\ \OpSem\\
	&(X,P_X)&\longmapsto&\ ( \calf(X),P_{\calf(X)}).
\end{aligned}
\end{equation}

\medskip

Now, we are going to show the main result of this subsection.
\begin{prop}\label{Prop: Free object functor from operated sets to operated semigroups}
	The   functor $\calf^{\OpSem}_{\OpSet}:\OpSet \rightarrow \OpSem $ is left adjoint to the forgetful functor $\calu_{\OpSet}^{\OpSem}:\OpSem\rightarrow \OpSet$.
\end{prop}

\begin{proof}  Let $(T, P_T)$ be an operated semigroup.
Let $$\theta:   (X,P_X)\to \calu_{\OpSet}^{\OpSem}(T,P_T))=(T, P_T)$$ be a morphism of operated sets.
We will construct a morphism of operated semigroups
$$\tilde{\theta}: \calf^{\OpSem}_{\OpSet}(X,P_X) \to (T,P_T)$$
such that $\tilde{\theta}\circ i=\theta$. Moreover, it is unique.

By the universal property of free semigroup construction, we have a homomorphism of semigroups
$\tilde{\theta}_0: \calf_0(X)=\cals(X)\rightarrow T$ extending $\theta$.  Now consider the map
$\bar{\theta}_1: X\sqcup \left \lfloor \cals(X) \backslash X \right \rfloor\to T$ by sending any element $x\in X$
		to $\theta(x)$ and $\left \lfloor u\right \rfloor$ to $P_T(\tilde{\theta}_0(u))$ for arbitrary $u\in \cals(X) \backslash X$, which induces
$\tilde{\theta}_1: \calf_1(X)=\cals( X\sqcup \left \lfloor \cals(X) \backslash X\right \rfloor)\to T$ such that $\tilde{\theta}_1 \circ i_{0,1}=\tilde{\theta}_0$.

 By iterating this process, we   found  a series of homomorphisms of semigroups $\tilde{\theta}_k: \calf_k(X)\rightarrow T, k\geqslant 0 $ which fit into the commutative diagram:
$$
\xymatrixrowsep{3pc}\xymatrixcolsep{4pc}\xymatrix{X\ar[r]^{\theta}\ar@{^{(}->}[d] & T&  \\
		\calf_0(X)\ar@{-->}[ur]_<(.5){\exists \tilde{\theta}_0}\ar@{^{(}->}[r]^{i_{0,1}} &\calf_1(X)\ar@{-->}[u]_{\exists \tilde{\theta}_1}\ar@{^{(}->}[r]^{i_{1,2}}&\calf_2(X)\ar@{-->}[ul]_<(.2){ \exists  \tilde{\theta_2}}\ar@{^{(}->}[r]^{i_{2,3}} &\cdots\ar@{^{(}->}[r]&\calf(X)\ar@{-->}[lllu]_{\exists  \tilde{\theta}}}\nonumber
$$
  It follows that there exists a homomorphism of semigroups $ \tilde{\theta} $ from $ \calf(X) $ to $ T $ and
  by construction $\tilde{\theta} $ is a morphism of operated semigroups.

  The uniqueness of $\tilde{\theta}$ is clear.

  This completes the proof.
\end{proof}

\subsection{Free object functor  from operated semigroups to operated monoids }\

In this subsection, we will construct the  functor  $ \calf^{\OpMon}_{\OpSem}:\OpSem\rightarrow \OpMon $ left adjoint to the forgetful functor $ \calu^{\OpMon}_{\OpSem}:\OpMon\rightarrow \OpSem $.

  Let $ (T,P_T) $ be an operated semigroup, $ \calf_0(T):= \tilde{\calm}(T\sqcup \varnothing)=\tilde{\calm}(T)$ be the monoid defined before. By Lemma~\ref{Lem: unique expression in Mtilde}, the natural inclusion $T\hookrightarrow T\sqcup  \left \lfloor \calf_0(T)\backslash T\right \rfloor$ induces an   injective monoid  homomorphism
$$ i_{0,1}: \calf_0(T) =\tilde{\calm}(T) \hookrightarrow \calf_1(T):=\tilde{\calm}(T\sqcup  \left \lfloor \calf_0(T)\backslash T\right \rfloor).
$$  For $ n\geq2 $, by induction starting from a series of  injective monoid  homomorphisms
$$\calf_0(T) )\overset{i_{0,1}}{\hookrightarrow}\calf_1(T)\overset{i_{1,2}}{\hookrightarrow} \cdots \overset{i_{n-3,n-2}}{\hookrightarrow} \calf_{n-2}(T)\overset{i_{n-2,n-1}}{\hookrightarrow}\calf_{n-1}(T)
$$ with $\calf_{k}(T) =\tilde{\calm}(T\sqcup  \left \lfloor \calf_{k-1}(T)\backslash T\right \rfloor)$ for all $1\leq k\leq n-1$,
one   imposes
\begin{equation}\label{key}
	\nonumber \calf_{n}(T):=\tilde{\calm}(T\sqcup  \left \lfloor \calf_{n-1}(T)\backslash T\right \rfloor).
\end{equation}
Again by Lemma~\ref{Lem: unique expression in Mtilde},    from the injection
\begin{equation}\label{key}
	\Id_T\sqcup \left \lfloor i_{n-2,n-1}|_{\calf_{n-2}(T)\backslash T}\right \rfloor:T\sqcup\left \lfloor \calf_{n-2}(T)\backslash T\right \rfloor\hookrightarrow T\sqcup\left \lfloor \calf_{n-1}(T)\backslash T\right \rfloor,\nonumber
\end{equation}
one can build an injective monoid  homomorphism
$$i_{n-1,n}: 	\calf_{n-1}(T) =\tilde{\calm}(T\sqcup  \left \lfloor \calf_{n-2}(T)\backslash T\right \rfloor) \hookrightarrow  \calf_{n}(T) =\tilde{\calm}(T\sqcup  \left \lfloor \calf_{n-1}(T)\backslash T\right \rfloor).
$$

Now we impose  $\calf(T):=\varinjlim \calf_{n}(T)$. Notice that there are natural induced injective monoid homomorphisms $ i_n:\calf_{n}(T)\hookrightarrow\calf(T),n\geq 0 $ and $i:T\hookrightarrow\calf(T)$.
Hence, we obtain  an operated monoid
\begin{equation}\label{key}
	\nonumber \left( \calf(T), P_{\calf(T)}(u):=
	\left\{
	\begin{array}{lcl}
		i(P_T(t)) &\text{if} &u=i(t), t \in T \\
		 i_{k+1}(\left\lfloor u' \right\rfloor) &\text{if} &u \notin  i(T),\ \text{but}\ u=i_k(u'), \ \text{for some}\ u'\in \calf_k(T)
	\end{array}
	\right.  \right).
\end{equation}
 It is ready to see that this forms a functor
\begin{equation}\label{key}
	\nonumber
	\begin{aligned}
		\calf^{\OpMon}_{\OpSem}:\  &\OpSem&\longrightarrow &\ \OpMon\\
		&(T,P_T)&\longmapsto&\ ( \calf(T),P_{\calf(T)}).
	\end{aligned}
\end{equation}
and it
  is left adjoint to the forgetful functor $ \calu^{\OpMon}_{\OpSem}:\OpMon\rightarrow \OpSem $. The details are left to the reader.

\subsection{Composites of free object functors} \


 It is easy to check that the composite  $\calf^{\OpSem}_{\Set}=\calf^{\OpSem}_{\Sem}\circ\calf^{\Sem}_{\Set}$  gives exactly the free operated semigroup generated by a set,  whose construction firstly appeared in \cite{Guo09a}.
 Although $\calf^{\OpSem}_{\Set}$ is naturally isomorphic to $ \calf^{\OpSem}_{\OpSet}\circ\calf^{\OpSet}_{\Set}$, the construction of the latter seems to be much more complicated.

 Similarly,  $\calf^{\OpMon}_{\Set}=\calf^{\OpMon}_{\Mon}\circ\calf^{\Mon}_{\Set}$ is exactly the free operated monoid generated by a set as constructed  in \cite{Guo09a}, whose another realisation $\calf^{\OpMon}_{\OpSet}\circ\calf^{\OpSet}_{\Set}$
  is  much more involved, however.

 In the sequel, following \cite{Guo09a},  denote $$(\frakS(X),P_{\frakS(X)}):=\calf^{\OpSem}_{\Set}(X)=\calf^{\OpSem}_{\Sem}(\cals(X))$$
 and
 $$  (\frakM(X),P_{\frakM(X)}):=\calf^{\OpMon}_{\Set}(X)=\calf^{\OpMon}_{\Mon}(\calm(X)). $$

\section{Free object functors III}


Parallel with the last section, we will  consider the back face of the diagram in Introduction:
  $$\xymatrixcolsep{5pc}\xymatrixrowsep{5pc}
\xymatrix{ \OpVect \ar@<1ex>[r]^{\calf^{\OpAlg}_{\OpVect}}    \ar@<1ex>[d]^{\calu^{\OpVect}_{\Vect}}
& \OpAlg  \ar@<1ex>[d]^{\calu^{\OpAlg}_{\Alg}}  \ar@<1ex>[l]^{\calu_{\OpVect}^{\OpAlg}}  \ar@<1ex>[r]^{\calf^{\uOpAlg}_{\OpAlg}}
& \uOpAlg \ar@<1ex>[l]^{\calu_{\OpAlg}^{\uOpAlg}}   \ar@<1ex>[d]^{\calu^{\uOpAlg}_{\uAlg}} \\
  \Vect  \ar@<1ex>[r]^{\calf^{\Alg}_{\Vect}}     \ar@<1ex>[u]^{\calf_{\Vect}^{\OpVect}}
&  \Alg   \ar@<1ex>[l]^{\calu_{\Vect}^{\Alg}}\ar@<1ex>[r]^{\calf^{\uAlg}_{\Alg}}  \ar@<1ex>[u]^{\calf_{\Alg}^{\OpAlg}}
&\uAlg \ar@<1ex>[u]^{\calf_{\uAlg}^{\uOpAlg}}  \ar@<1ex>[l]^{\calu_{\Alg}^{\uAlg}}}$$
Our goal in this section is to construct the free object functors $\calf^*_*$ which are left adjoint to the corresponding forgetful functors $\calu^*_*$.

Since all the proofs in this section are similar to those in the previous sections, we omit all of them.




 \subsection{Free object functor from vector spaces  to operated vector spaces}\

	Let $V$ be a $\bk$-space with a basis $X$. Extend the map $\lfloor ~ \rfloor:X\rightarrow\lfloor X \rfloor$ to a linear map $\lfloor ~ \rfloor:V\rightarrow\lfloor V\rfloor:=\bk\lfloor X \rfloor$, where by abuse of notations, we  use the same symbol for the induced linear map. It is easy to see that this definition is independent of the choice of basis. Similarly, we can define $\left \lfloor V\right \rfloor^{(n)}$ and $\left \lfloor v\right \rfloor^{(n)}$ for $n\geq 0$ and  $v\in V$.
	Define  $\calv:=\bigoplus_{n \geq 0}\left \lfloor V\right \rfloor^{(n)}$ and an operator $P_{\calv}$ on $\calv$ mapping $\lfloor v\rfloor^{(n)}$  to  $\lfloor v\rfloor^{(n+1)}$ for  all  $v\in V$ and $n\geq 0$. Then  the pair $\calf^{\OpVect}_{\Vect}(V):=(\calv, P_\calv)$ forms an operated $\bk$-space.   In fact,   it can be seen that  $ (\calv, P_\calv)=\calf^{\OpVect}_{\OpSet}\circ \calf^{\OpSet}_{\Set}(X)$.

For a linear map $ f: V\to W$, introduce $\calf^{\OpVect}_{\Vect}(f): (\calv, P_\calv) \to (\calw, P_\calw)$ to be the linear map sending $\lfloor v\rfloor^{(n)}$   to $\lfloor f(v)\rfloor^{(n)}$ for  all $n\geq 0$ and   $v\in V$.  It is easy to see that $\calf^{\OpVect}_{\Vect}$ is a functor from $\Vect $ to $\OpVect $.

		It is ready to see that the   functor $\calf^{\OpVect}_{\Vect}:\Vect \rightarrow \OpVect$ is left adjoint to the forgetful functor $\calu_{\Vect}^{\OpVect}:\OpVect\rightarrow \Vect$, hence giving the free operated $\bk$-space generated by a $\bk$-space.


\subsection{Free object functor from algebras to operated algebras}\

In this subsection, we will construct the functor   $ \calf^{\OpAlg}_{\Alg}:\Alg\rightarrow \OpAlg $ left adjoint to  the forgetful functor $ \calu_{\Alg}^{\OpAlg}:\OpAlg\rightarrow \Alg $.

\begin{defn}
	Let $ V$ be a $\bk$-space, $ A $ be a $\bk$-algebra. We define the following quotient $\bk$-algebra:
	$$\tilde{\bar{\TT}}(A\oplus V):=\bar{\TT}(A\oplus V)\slash\left\langle a\otimes b-a\cdot b~|~\forall a,b\in A\right\rangle ,$$
	where $ \otimes $ and $  \cdot $ are respectively the multiplications in  $\bar{\TT}(A\oplus V)$ and $A$. In particular, if $ V=0 $, then $ \tilde{\bar{\TT}}(A)=\bar{\TT}(A)\slash\left\langle a\otimes b-a\cdot b~|~\forall a,b\in A\right\rangle=A $.
\end{defn}

\begin{lemma}\label{Lem: injective of algebra}
	Let $V\hookrightarrow V'$ be an injective linear map, $A\hookrightarrow B$ be a monomorphism of $\bk$-algebras. Then the induced homomorphism of algebras
	$$\tilde{\bar{\TT}}(A\oplus V)\rightarrow \tilde{\bar{\TT}}(B\oplus V')$$
	is injective.
	\end{lemma}

Let $A$ be a $\bk$-algebra and $\calf_0(A):=A$. The inclusion into the first component $ A\hookrightarrow A\oplus  \left \lfloor A\right \rfloor$ induces an injective algebra homomorphism
$$
i_{0,1}: \calf_0(A)=A{\hookrightarrow}\calf_1(A):=\tilde{\bar{\TT}}(A\oplus \left \lfloor A\right \rfloor)
.$$

For $ n\geq2 $, assume  that we have a series of monomorphisms of algebras
$$\calf_0(A) \overset{i_{0,1}}{\hookrightarrow}\calf_1(A)\overset{i_{1,2}}{\hookrightarrow} \cdots \overset{i_{n-3,n-2}}{\hookrightarrow} \calf_{n-2}(A)\overset{i_{n-2,n-1}}{\hookrightarrow}\calf_{n-1}(A).
$$
with $\calf_{k}(A):=\tilde{\bar{\TT}}(A\oplus \left \lfloor \calf_{k-1}(A)\right \rfloor)$ for $1\leq k\leq n-1$.
 We define the $\bk$-algebra
$$
\calf_{n}(A):=\tilde{\bar{\TT}}(A\oplus \left \lfloor \calf_{n-1}(A)\right \rfloor).
$$
The natural  injection
$$	\Id_A\oplus \left \lfloor i_{n-2,n-1}\right \rfloor :A\oplus\left \lfloor \calf_{n-2}(A)\right \rfloor\hookrightarrow A\oplus\left \lfloor \calf_{n-1}(A)\right \rfloor\nonumber$$
induces a $\bk$-algebra injective homomorphism
$$
i_{n-1,n}:		\calf_{n-1}(A)=\tilde{\bar{\TT}}(A\oplus  \left \lfloor \calf_{n-2}(A)\right \rfloor){\hookrightarrow} \calf_{n}(A)=\tilde{\bar{\TT}}(A\oplus  \left \lfloor \calf_{n-1}(A)\right \rfloor).
$$
Denote $ \calf(A)=\varinjlim \calf_{n}(A)=\bigcup_{n\geq 0}\calf_{n}(A)$ and an operator $P_{\calf}$ over $\calf$ sending $v\in \calf$ to $\left \lfloor v\right \rfloor$ . Thus, we have constructed the free object functor:
$$
\begin{aligned}
	\calf_{\Alg}^{\OpAlg}: &\Alg&\longrightarrow &\ \OpAlg\\
	&A&\longmapsto&\ ( \calf(A),P_{\calf(A)}),
\end{aligned}
$$
which is left adjoint to the forgetful functor $ \calu^{\OpAlg}_{\Alg}$.

\subsection{Free object functor from unital algebras to unital operated algebras}\label{Free object functor from unital algebras to unital operated algebras}\

\begin{defn}
	Let $A$ be a unital $\bk$-algebra and $V$ a $\bk$-space. We define the following quotient unital $\bk$-algebra
	$$
	\hat{\TT}(A\oplus V):=\TT(A\oplus V)\slash\left\langle a\otimes b-a\cdot b,\ 1_A-1_{\TT(A\oplus V)} ~|~\forall a,b\in A\right\rangle
	.$$
	
	In particular, if $ V=0 $, then $ \hat{{\TT}}(A)={\TT}(A)\slash\left\langle a\otimes b-a\cdot b~|~\forall a,b\in A\right\rangle=A $.
\end{defn}

\begin{lemma}\label{Lem: injective of unital algebra}
	Let $V\hookrightarrow V'$ be an injective linear map, $A\hookrightarrow B$ be a monomorphism of unital $\bk$-algebras. Then the induced homomorphism
	$$\hat{\TT}(A\oplus V)\rightarrow \hat{\TT}(B\oplus V')$$
	is injective.
\end{lemma}

Let $A$ be a unital $\bk$-algebra and $\calf_0(A):= \hat{\TT}(A)=A$. By Lemma~\ref{Lem: injective of unital algebra}, the inclusion $ A\hookrightarrow A\oplus  \left \lfloor A\right \rfloor$ induces a monomorphism
$$
i_{0,1}:	 \calf_0(A)=\hat{\TT}(A){\hookrightarrow}\calf_1(A):=\hat{\TT}(A\oplus \left \lfloor A\right \rfloor)
$$
of unital algebras.

 For $ n\geq2 $, by induction starting from a series of monomorphisms of unital algebras
$$\calf_0(A)\overset{i_{0,1}}{\hookrightarrow}\calf_1(A)\overset{i_{1,2}}{\hookrightarrow} \cdots \overset{i_{n-3,n-2}}{\hookrightarrow} \calf_{n-2}(A)\overset{i_{n-2,n-1}}{\hookrightarrow}\calf_{n-1}(A)
$$
with $\calf_{k}(A)=\hat{\TT}(A\oplus \left \lfloor \calf_{k-1}(A)\right \rfloor) $ for $1\leq k\leq n-1$.
 Then we define the unital algebra
\begin{equation}\label{key}
	\nonumber \calf_{n}(A):=\hat{\TT}(A\oplus \left \lfloor \calf_{n-1}(A)\right \rfloor).
\end{equation}
The natural  injection
$$	\Id_{A}\oplus \left \lfloor i_{n-2,n-1}\right \rfloor :A\oplus\left \lfloor \calf_{n-2}(A)\right \rfloor\hookrightarrow A\oplus\left \lfloor \calf_{n-1}(A)\right \rfloor\nonumber$$
induces an injective algebra  homomorphism
$$
i_{n-1,n}:		\calf_{n-1}(A)=\hat{\TT}(A\oplus  \left \lfloor \calf_{n-2}(A)\right \rfloor) {\hookrightarrow} \calf_{n}(A)=\hat{\TT}(A\oplus  \left \lfloor \calf_{n-1}(A)\right \rfloor).
$$
Denote $ \calf(A)=\varinjlim \calf_{n}(A)=\bigcup_{n\geq 0}\calf_{n}(A)$ and define an operator $P_{\calf(A)}$ over $\calf(A)$ by sending $v\in \calf(A)$ to $\left \lfloor v\right \rfloor$. Thus, we have constructed the free object functor:
$$
\begin{aligned}
	\calf_{\uAlg}^{\uOpAlg}: &\uAlg&\longrightarrow &\ \uOpAlg\\
	&A&\longmapsto&\ ( \calf(A),P_{\calf(A)}).
\end{aligned}
$$
left adjoint to the forgetful functor $ \calu^{\uOpAlg}_{\uAlg}$.

\subsection{Free object functor from operated vector spaces  to operated algebras}\

In this subsection, we will construct the free object functor $ \calf^{\OpAlg}_{\OpVect} $ from $\OpVect$ to $\OpAlg$ which is left adjoint to the forgetful functor $\calu^{\OpAlg}_{\OpVect} $ from $\OpAlg$ to $\OpVect$.

\par Let $ (V,P_V)$ be an operated vector space.   As $\bar{\TT}(V)$ has a canonical splitting $\bar{\TT}(V)=V\oplus V^{\otimes 2}\oplus  \cdots$ and we will write $\bar{\TT}(V)\slash V=V^{\otimes 2}\oplus  \cdots$ in the sequel. The reader should understand the notations $\calf_{n}(V)/V$,  $\tilde{\TT}(A)\slash A$ and $\calf_{n}(A)/A$ via the canonical splitting.

The inclusion $ i_0:V\hookrightarrow V\oplus  \left \lfloor \bar{\TT}(V)\slash V\right \rfloor$ induces a monomorphism of algebras:
$$
i_{0,1}:	 \calf_0(V):=\bar{\TT}(V){\hookrightarrow}\calf_1(V):=\bar{\TT}(V\oplus \left \lfloor \bar{\TT}(V)\slash V\right \rfloor).
$$
 For $ n\geq2 $, suppose  that  a series of monomorphisms of algebras
$$\calf_0(V) \overset{i_{0,1}}{\hookrightarrow}\calf_1(V)\overset{i_{1,2}}{\hookrightarrow} \cdots \overset{i_{n-3,n-2}}{\hookrightarrow} \calf_{n-2}(V)\overset{i_{n-2,n-1}}{\hookrightarrow}\calf_{n-1}(V),
$$
has been obtained. Then we define
$$
\calf_{n}(V):=\bar{\TT}(V\oplus  \left \lfloor \calf_{n-1}(V)\slash V\right \rfloor).
$$
The injection
$$
\Id_V\oplus \left \lfloor i_{n-2,n-1}|_{\calf_{n-2}(V)\slash V}\right \rfloor:V\oplus\left \lfloor \calf_{n-2}(V)\slash V\right \rfloor\hookrightarrow V\oplus\left \lfloor \calf_{n-1}(V)\slash V\right \rfloor
$$
induces a $\bk$-algebra injection
$$
i_{n-1,n}:		\calf_{n-1}(V):=\bar{\TT}(V\oplus  \left \lfloor \calf_{n-2}(V)\slash V\right \rfloor){\hookrightarrow} \calf_{n}(V):=\bar{\TT}(V\oplus  \left \lfloor \calf_{n-1}(V)\slash V\right \rfloor).
$$
Notice that  the subspace $V$ has a canonical complement in each $\calf_k(V)$; for instance, the complement of $V$ in $\calf_0(V)=\bar{\TT}(V)$ is exactly $\bigoplus_{n\geqslant 2}V^{\ot n}$. In practice, we always identify $\calf_k(V)/V$ with this complement.

By the limit construction, there are injective homomorphisms $i_n: \calf_n(V)\hookrightarrow \calf(V), n\geqslant 0$ and $i: V\hookrightarrow \calf(V)$.
Then, we define an operated $\bk$-algebra
$$
	\nonumber \left( \calf(V):=\varinjlim \calf_{n}(V), P_{\calf(V)}(u):=
	\left\{
	\begin{array}{lcl}
		i(P_V(v)) &\text{if} &u=i(v) \ \text{for some}\ v \in V \\
		i_{k+1}(\left\lfloor u'\right\rfloor) &\text{if} & u\notin i(V),\ \text{but}\  u=i_k(u')\ \text{for some}\ u'\in \calf_k(V)
	\end{array}
	\right.  \right).
$$
Thus we get a functor
$$\begin{aligned}
		\calf^{\OpAlg}_{\OpVect}: &\OpVect&\longrightarrow &\OpAlg\\
		&(V,P_V)&\longmapsto&( \calf(V),P_{\calf(V)}).
	\end{aligned}
$$

We obtain the following result:
\begin{prop}
	The functor $	\calf^{\OpAlg}_{\OpVect}$ is the left adjoint of the forgetful functor $\calu^\OpAlg_{\OpVect}:\OpAlg\rightarrow \OpVect$.
\end{prop}

\subsection{Free object functor from operated algebras to unital operated algebras}\

In this subsection, we will construct the free object functor $\calf_{\OpAlg}^{\uOpAlg}$ from $\OpAlg$ to $\uOpAlg$ which is left adjoint to the forgetful functor $\calu^{\uOpAlg}_{\OpAlg}$ from $\uOpAlg$ to $\OpAlg$.

\begin{defn}
	Let $ V$ be a $\bk$-space, $ A $ be a $\bk$-algebra. We define the following quotient unital $\bk$-algebra:
	$$
	\tilde{\TT}(A\oplus V):=\TT(A\oplus V)\slash\left\langle a\otimes b-a\cdot b~|~\forall a,b\in A\right\rangle
	,$$
	where $ \otimes $ and $  \cdot $ are respectively the multiplications in $\TT(A\oplus V)$  and $A$. In particular, if $ V=0 $ then $ \tilde{\TT}(A\oplus V)= \TT(A)\slash\left\langle a\otimes b-a\cdot b~|~\forall a,b\in A\right\rangle=\bk\oplus A$.
\end{defn}

\begin{lemma}\label{Lem: injective of algebra}
	Let $V\hookrightarrow V'$ be an injective linear map, $A\hookrightarrow B$ be an injective homomorphism of $\bk$-algebras. Then the induced homomorphism of unital algebras
	$$\tilde{{\TT}}(A\oplus V)\rightarrow \tilde{{\TT}}(B\oplus V')$$
	is injective.
\end{lemma}

Let $ (A,P_A) $ be an operated $\bk$-algebra and $\calf_0(A):=\tilde{\TT}(A)$. Deduced from Lemma~\ref{Lem: injective of algebra}, the inclusion $ i_0:A\hookrightarrow A\oplus  \left \lfloor \tilde{\TT}(A)\slash A\right \rfloor$ induces an injective homomophism of unital algebras
$$
i_{0,1}: \calf_0(A)=\tilde{\TT}(A){\hookrightarrow}\calf_1(A):=\tilde{\TT}(A\oplus \left \lfloor \tilde{\TT}(A)\slash A\right \rfloor).
$$
 For $ n\geq2 $, assume that a series of injective homomorphisms of unital algebras
$$\calf_0(A) \overset{i_{0,1}}{\hookrightarrow}\calf_1(A)\overset{i_{1,2}}{\hookrightarrow} \cdots \overset{i_{n-3,n-2}}{\hookrightarrow} \calf_{n-2}(A)\overset{i_{n-2,n-1}}{\hookrightarrow}\calf_{n-1}(A),
$$
with $\calf_{k}(A)=\tilde{\TT}(A\oplus  \left \lfloor \calf_{k-1}(A)\slash A\right \rfloor)$ for all $1\leqslant k\leqslant n-1$,
has been constructed. Then we define
$$ \calf_{n}(A):=\tilde{\TT}(A\oplus  \left \lfloor \calf_{n-1}(A)\slash A\right \rfloor).
$$
By Lemma~\ref{Lem: injective of algebra}, the injection
$$
\Id_A\oplus \left \lfloor i_{n-2,n-1}|_{\calf_{n-2}(A)\slash A}\right \rfloor:A\oplus\left \lfloor \calf_{n-2}(A)\slash A\right \rfloor\hookrightarrow A\oplus\left \lfloor \calf_{n-1}(A)\slash A\right \rfloor
$$
 induces an injective homomorphism of unital $\bk$-algebras:
$$
i_{n-1,n}:		\calf_{n-1}(A)=\tilde{\TT}(A\oplus  \left \lfloor \calf_{n-2}(A)\slash A\right \rfloor){\hookrightarrow} \calf_{n}(A)=\tilde{\TT}(A\oplus  \left \lfloor \calf_{n-1}(A)\slash A\right \rfloor).
$$

By the limit construction, there are injective homomorphisms $i_n: \calf_n(A)\hookrightarrow \calf(A), n\geqslant 0$ and $i: A\hookrightarrow \calf(A)$.
Finally, we define an operated untial $\bk$-algebra
\begin{equation}
	\nonumber \left( \calf(A):=\varinjlim \calf_{n}(A), P_{\calf(A)}(u):=
	\left\{
	\begin{array}{lcl}
		i(P_A(a)) &\text{if} &u=i(a) \ \text{for some }\  a\in {A} \\
		 i_{k+1}(\left\lfloor u'\right\rfloor) &\text{if} &u\notin i(A), \text{but}\ u=i_k(u') \ \text{for some}\ u' \in \calf_k(A)
	\end{array}
	\right. \right).
\end{equation}

Thus we obtain a functor
\begin{equation}
	\nonumber
	\begin{aligned}
		\calf_{\OpAlg}^{\uOpAlg}: &\OpAlg&\longrightarrow &\ \uOpAlg\\
		&(A,P_A)&\longmapsto&\ ( \calf(A),P_{\calf(A)})
	\end{aligned}
\end{equation}
and we can prove the following result:
\begin{prop}
	The  functor $\calf_{\OpAlg}^{\uOpAlg}$ is the left adjoint of the forgetful functor $\calu^{\uOpAlg}_{\OpAlg}:\uOpAlg\rightarrow \OpAlg$, thus it is the free object functor from $\OpAlg$ to $\uOpAlg$.
\end{prop}

\subsection{Compositions of free object functors}\

It is easy to check that the composite  $\calf^{\uOpAlg}_{\Set}:=\calf^{\uOpAlg}_{\OpMon}\circ\calf^{\OpMon}_{\Set}=\calf_{\uAlg}^{\uOpAlg}\circ\calf_{\Set}^{\uAlg}$  provides   exactly the free unital operated $\bk$-algebra generated by a set,  whose construction firstly appeared in \cite{Guo09a}; for a set $X$, $\calf_{\Set}^{\uOpAlg}(X)=\bk \frakM(X)$.

Although $\calf^{\uOpAlg}_{\Set}$ has many other constructions, as indicated by the cubic diagram in Introduction,  these different  constructions  seem to be much more complex.

\begin{remark} From the analysis of the first three sections, it is interesting to notice that except the four squares in the front  face and the back  face, all other squares of free object functors are actually commutative.

\end{remark}

To conclude this section, we include a result which  considers the action of $\calf_{\uAlg}^{\uOpAlg}$ on a unital  algebra $A=\TT(V)\slash I_A$.
\begin{prop}\label{Prop: algebra ideal considered as operated ideal}
	Let $A=\TT(V)\slash I_A$ be a unital algebra. Then we have:
	$$\calf_{\uAlg}^{\uOpAlg}(A)=\calf_{\Vect}^{\uOpAlg}(V)\slash \left\langle I_A \right\rangle_\uOpAlg,$$
	where $I_A$ is considered as a subset in $\calf_{\Vect}^{\uOpAlg}(V)$ via the canonical inclusion $$j: \TT(V) =\calf_0(\TT(V)) \hookrightarrow \varinjlim{\calf}_n(\TT(V))=\calf_{\uAlg}^{\uOpAlg}(\TT(V))=\calf_{\Vect}^{\uOpAlg}(V).$$
\end{prop}

\begin{proof}

Write canonical surjections  $\pi: \TT(V)\to A$   and $p: \calf_{\Vect}^{\uOpAlg}(V)  \to \calf_{\Vect}^{\uOpAlg}(V)\slash \left\langle I_A \right\rangle_\uOpAlg $. Since $I_A\subseteq \TT(V)\cap \left\langle I_A \right\rangle_\uOpAlg$, the inclusion $j$ induces a homomorphism of unital algebras
$$i: A\to \calf_{\Vect}^{\uOpAlg}(V)\slash \left\langle I_A \right\rangle_\uOpAlg.$$
We have obviously $p\circ j=i\circ \pi.$

 Let $(B,P)$ be a unital operated algebra and  $f:A\rightarrow B$ be  a homomorphism of unital algebras.  We will show that there exists a unique morphism of unital operated algebras
 $$\tilde{f}: \calf_{\Vect}^{\uOpAlg}(V)\slash \left\langle I_A \right\rangle_\uOpAlg\to (B, P)$$ such that
 $\tilde{f}\circ i=f$.

 Consider the following diagram:
$$
 \xymatrix{\TT(V) \ar[rr]^-{j}    \ar[dd]_{\pi} \ar[dr]^{f\circ \pi}
 	&  &\calf_{\uAlg}^{\uOpAlg}(\TT(V))=\calf_{\Vect}^{\uOpAlg}(V) \ar[dd]^{p}  \ar@{-->}[dl]_-{\hat{f}}\\
 	&(B,P)&\\
 	A \ar[rr]^-{i}\ar[ur]^{f}
 	&
 	&{\calf_{\Vect}^{\uOpAlg}(V)\slash \left\langle I_A \right\rangle_\uOpAlg} \ar@{-->}[ul]_{\tilde{f}} }	$$
  The composition $f\circ \pi$ is a homomorphism of unital algebras and sends $I_A$ to zero. By the universal property of $\calf_{\uAlg}^{\uOpAlg}(\TT(V))=\calf_{\Vect}^{\uOpAlg}(V)$, there is a unique morphism $\hat{f}$ of unital operated algebras from $\calf_{\Vect}^{\uOpAlg}(V)$ to $(B,P)$  such that $\hat{f}\circ j=f\circ \pi$. As $j$ is injective,  we also have $I_A\subseteq\ker(\hat{f})$, thus the operated ideal $\left\langle I_A \right\rangle_\uOpAlg $ generated by $I_A$ is also contained in $\ker(\hat{f})$. So $\hat{f}$ factors through $\calf_{\Vect}^{\uOpAlg}(V)\slash \left\langle I_A \right\rangle_\uOpAlg$, which induces  a   morphism $\tilde{f}$ of unital operated algebras from $\calf_{\Vect}^{\uOpAlg}(V)\slash \left\langle I_A \right\rangle_\uOpAlg$ to $(B,P)$ so that $\tilde{f}\circ p=\hat{f}$. Since  $i\circ\pi=p\circ j$, we have
$$\tilde{f}\circ i\circ\pi=\tilde{f}\circ p\circ j=\hat{f}\circ j=f\circ \pi.$$
Hence, $\tilde{f}\circ i=f$, by the surjectivity of  $\pi$.

For the  uniqueness, assume that there is another morphism of unital operated algebras
 $$\tilde{g}: \calf_{\Vect}^{\uOpAlg}(V)\slash \left\langle I_A \right\rangle_\uOpAlg\to (B, P)$$ such that
 $\tilde{g}\circ i=f$.  So $\hat{g}:=\tilde{g}\circ p$ satisfies
 $$\hat{g}\circ j=\tilde{g}\circ p \circ j=\tilde{g}\circ i\circ\pi=f\circ \pi.$$
 By the uniqueness of $\hat{f}$, we  get $\hat{g}=\hat{f}$ and thus $\tilde{g}=\tilde{f}$ as $p$ is surjective.
 \end{proof}



\section{Free $\Phi$-algebras via  universal algebra}
Free objects in operated contexts have been constructed in the previous sections. In this section, we will consider an operated algebra  whose  operator  satisfies extra relations.  It will be seen that free operated algebras with prescribed relations on  the operator    are given by    quotients of free operated algebras.  Our method is in the realm of    universal algebra following \cite{Guo09a}.

Throughout this section, let $X$ be a set.  Recall that $\bk\frakM(X)$   denotes  the free unital operated algebra generated by $X$.

\begin{defn}[{\cite{GSZ, GGSZ, GaoGuo17}}]
	 Let $\phi(x_1, \dots, x_n)\in \bk\frakM(X) $ with $ n\geqslant 1, x_1, \dots, x_n\in X$. We call $\phi(x_1,\dots,x_n)=0$ (or just $\phi(x_1,\dots,x_n)$) an operated polynomial identity (aka  OPI ).
\end{defn}
Let $\phi=\phi(x_1,\dots,x_n)$ be an  OPI  in $\bk\frakM(X)$ and  $(A,P)$ be  a  unital operated algebra. Let $\theta: X\rightarrow A$ be a map with $r_i= \theta(x_i), 1\leqslant i\leqslant n$.
By the universal property of $\bk\frakM(X)$, there is a unique morphism of  unital operated algebras $\tilde{\theta}: \bk\frakM(X)\rightarrow A.$  We denote :
\[ \phi(r_1,\dots,r_n):=\tilde{\theta}(\phi(x_1,\dots,x_n)). \]

\begin{defn}[{\cite{GSZ, GGSZ, GaoGuo17}}]  Let $\phi(x_1,\dots,x_n)$ be an  OPI. A unital operated algebra $(A,P)$ is said to satisfy the OPI $\phi(x_1,\dots,x_n)$ if
	\[\phi(r_1,\dots,r_n)=0 \ \ \  \mbox{for all}\  r_1,\dots,r_n\in A.\]
	In this case, $(A,P)$ is called a   $\phi$-algebra and $P$ is called a $\phi$-operator.
	
	Generally, for a family of OPIs $\Phi\subset \bk\frakM(X)$, we call a  unital operated algebra $(A,P)$ a   $\Phi$-algebra if it is a   $\phi$-algebra for any $\phi\in \Phi$.
	Denote the full subcategory of   $\Phi$-algebras in $\uOpAlg$ by $\Phi\zhx\Alg$. 
\end{defn}

\begin{defn}[{\cite{GSZ, GGSZ, GaoGuo17}}] Let $(A,P)$ be a unital operated algebra. An ideal $I$ of $A$ is called an operated ideal if it is closed under the action of the operator $P$.
	The operated ideal generated by a subset $S\subset A$ is denoted by $\left\langle S\right\rangle_\uOpAlg$.
\end{defn}
It is easy to see that the quotient of a  unital operated algebra by an operated ideal is naturally a unital operated algebra.

Let $\Phi$ be a set of OPIs in   $\bk\frakM(X)$.
For a unital algebra  $A$, define  $R_{\Phi}(A)\subseteq \calf^{\uOpAlg}_{\uAlg}(A)$ to be the following set
	\[R_\Phi(A):=\{\phi(u_1,\dots,u_n)\ |\ u_1,\dots,u_n\in \calf^{\uOpAlg}_{\uAlg}(A),\phi(x_1,\dots,x_n)\in \Phi\}.\]
 In particular, when $A=\bk \calm(Z)$ is the free unital   algebra generated by a set $Z$, we also use the notation $R_\Phi(Z):=R_\Phi(\bk \calm(Z))\subseteq \bk\frakM(Z)$ for short.
We also define another set $S_{\Phi}(Z)\subseteq\frakM(Z)$ by
\[S_{\Phi}(Z):=\{\phi(u_1,\dots,u_k)\ |\ u_1,\dots,u_k\in \frakM(Z),\phi(x_1,\dots,x_k)\in \Phi\}.\]

\begin{prop}[{see for example \cite[Proposition 1.3.6]{Cohn}}]\label{Prop: free Phi algebra generated by a set}
	Let   $\Phi\subseteq  \bk\frakM(X)$ be a set of OPIs. For a set $Z$, the quotient unital operated algebra  $\calf_{\Set}^{\Phi\zhx \Alg}(Z):=\calf^\uOpAlg_\Set(Z)\slash \left\langle R_\Phi(Z)\right\rangle_\uOpAlg$  is the free   $\Phi$-algebra generated by $Z$.
\end{prop}




\begin{defn}[\cite{GSZ, GGSZ, GaoGuo17}] An OPI $\phi(x_1,\dots,x_n)\in  \bk\frakM(X)$ is said to be multilinear if it is linear in each $x_i,~1\leq i\leq n$.

\end{defn}
\begin{remark} \label{Rem: in multilinear case, free Phi-algebra over  a given algebra via universal algebra}
	\begin{itemize}
\item[(a)] When   the set $\Phi$ consists of  multilinear OPIs,  for any given set $Z$,  $R_{\Phi}(Z)$ becomes a $\bk$-space and $R_{\Phi}(Z)=\bk S_{\Phi}(Z)$, thus the free  $\Phi$-algebra generated by $Z$  can also be written as $$\calf_{\Set}^{\Phi\zhx \Alg}(Z)=\calf^\uOpAlg_\Set(Z)\slash \left\langle S_\Phi(Z)\right\rangle_\uOpAlg.$$
\item[(b)] The multilinear property on OPIs is not strong in the sense that one can always use a  polarization  so long as the  OPIs in question are multilinear, at least when the base field is of characteristic zero; see \cite[Corollary 2.18]{GaoGuoR}. \textit{We will assume from now on that all OPIs are multilinear. }

\end{itemize} 	
\end{remark}

Similar to Proposition~\ref{Prop: free Phi algebra generated by a set}, we can    construct  the free $\Phi$-algebra generated by a given algebra as well.

\begin{prop}\label{Prop: free Phi-algebra over  a given algebra via universal algebra}
	Let    $\Phi\subseteq  \bk\frakM(X)$ be a set of OPIs. Given     a unital $\bk$-algebra $A$,  the quotient operated algebra $\calf_{\uAlg}^{\Phi\zhx \Alg}(A):=\calf^{\uOpAlg}_{\uAlg}(A)/{\left\langle R_{\Phi}(A)\right\rangle_\uOpAlg}$ is the free $\Phi$-algebra generated by the unital algebra $A$.
\end{prop}
\begin{proof} Let   $B$ be a   $\Phi$-algebra. Given a homomorphism of unital    algebras from $A$ to $\calu^{\Phi\zhx \Alg}_\uAlg(B)$,
by the universal property of   $\calf^{\uOpAlg}_\uAlg(A)$, we get a morphism of unital operated algebras from $\calf^{\uOpAlg}_\uAlg(A)$ to $B$.
Since $B$ is a  $\Phi$-algebra, any morphism of unital operated algebras from $\calf^{\uOpAlg}_\uAlg(A)$ to $B$	will factor through the $\Phi$-algebra $\calf^{\uOpAlg}_\uAlg(A)/\langle R_\Phi(A)\rangle_\uOpAlg$, thus we obtain a morphism of   $\Phi$-algebras from  $\calf^{\uOpAlg}_\uAlg(A)/\langle R_\Phi(A)\rangle_\uOpAlg$ to $B$ fulfilling the required universal property.  \end{proof}

	


\begin{prop}\label{Prop: from unital algebra to Phi algebra with relations}
	Let $X$ be a set and $\Phi\subseteq \bk\frakM(X)$ a system of OPIs. Let $A=\bk \calm (Z) \slash I_A$ be a unital algebra with generating set $Z$.  Then we have:
	$$ \calf_{\uAlg}^{\Phi\zhx\Alg}(A)=\bk\frakM(Z)\slash\left\langle  S_{\Phi}(Z)\cup I_A\right\rangle_\uOpAlg.$$

\end{prop}

\begin{proof}
	By Proposition~\ref{Prop: algebra ideal considered as operated ideal} and Remark~\ref{Rem: in multilinear case, free Phi-algebra over  a given algebra via universal algebra}, we have
	$$	\begin{array}{lcl}		\calf_{\uAlg}^{\Phi\zhx\Alg}(A)&=&\calf^{\uOpAlg}_{\uAlg}(A)\slash\left\langle R_{\Phi}(A)\right\rangle_\uOpAlg	\\
		&=&\calf^{\uOpAlg}_{\Vect}(\bk Z)\slash\left\langle R_{\Phi}(Z)\cup I_{A}\right\rangle_\uOpAlg\\
		&=&\bk\frakM(Z)\slash\left\langle S_{\Phi}(Z)\cup  I_{A}\right\rangle_\uOpAlg.\\
		\end{array}	$$ \end{proof}

\section{Gr\"obner-Shirshov bases for free $\Phi$-algebras over algebras}\

 In this section,  as a generalisation of the well developed theory of  Gr\"obner-Shirshov bases (aka GS bases)   for associative algebras  \cite{Shirshov,Buc,Green,BokutChen14,BokutChen20},   the  GS basis theory for  operated algebras  is recalled following \cite{BokutChenQiu, GSZ, GGSZ, GaoGuo17}.  Inspired by \cite{LeiGuo, GuoLi},  we will  also consider  the general problem of  GS bases  for free $\Phi$-algebra generated by a given algebra.


\begin{defn} [\cite{BokutChenQiu, GSZ, GGSZ, GaoGuo17}]
	Let $Z$ be a set, $\star$ a symbol not in $Z$.
	\begin{itemize}
		\item [(a)] Define $\frakM^\star(Z)$ to be the subset of $\frakM(Z\cup\star)$ consisting of elements with $\star$ occurring only once.
		\item [(b)] For $q\in\frakM^\star(Z)$ and $u\in   \frakM(Z)$,	we define $q|_{u}\in \frakM(Z)$ obtained by
		replacing the symbol $\star$ in $q$ by $u$.
		\item [(c)] For $q\in\frakM^\star(Z)$ and $s=\sum_ic_iu_i\in \bk \frakM(Z)$  with $c_i\in\bk$ and $u_i\in\frakM(Z)$, we define
		$$q|_s:=\sum_ic_iq|_{u_i}.$$
		\item [(d)]  An element $u\in\frakM(Z)$ is a subword of another element $w\in\frakM(Z)$ if $w=q|_{u}$ for some $q\in\frakM^\star(Z)$.
		\item [(e)] For $ f=\sum_ic_iq_i|_u\in\frakM(Z)$  with  $q_i\in\frakM^{\star} (Z)$, $c_i\in\bk$ and $u\in\frakM(Z)$, we define
		its substitution of $u$  by $s\in\frakM(Z)$ to be $$f|_{u\rightarrow s}:=\sum_ic_iq_i|_s.$$
	\end{itemize}
\end{defn}

\begin{defn} [\cite{BokutChenQiu, GSZ, GGSZ, GaoGuo17}]
	Let $Z$ be a set,  $\leq$ a linear order on $\frakM(Z)$ and $f \in \bk \frakM(Z)$.
	\begin{itemize}
		\item [(a)] Let $f\notin \bk$. The leading monomial of $f$ , denoted by $\bar{f}$, is the largest monomial appearing in $f$. The leading coefficient of $f$ , denoted by $c_f$, is the coefficient of $\bar{f}$ in $f$. We call $f$	monic with respect to $\leq$ if $c_f = 1$.
		\item [(b)] Let $f\in \bk$ (including the case $f=0$).  We define the leading monomial of $f$ to be $1$ and the	leading coefficient of $f$ to be $c_f=f$.
		\item [(c)] A subset $S\subseteq \bk\frakM(Z)$ is called monicized with respect to $\leq$,  if each nonzero element of $S$ has leading coefficient $1$. Obviously, each subset $S\subseteq \frakM(Z)$ can be made monicized  if we divide  each nonzero  element by    its  leading coefficient.
	\end{itemize}
\end{defn}
\begin{defn} [\cite{BokutChenQiu, GSZ, GGSZ, GaoGuo17}]
	Let $Z$ be a set. We deonte $u< v $ if $u\leq v$ but $u\neq v$ for an order $\leq$.
	\begin{itemize}	
		\item [(a)] A monomial order on $\calm(Z)$ is a well-order $\leq$ on $\calm(Z)$ such that
		$$  u < v \Rightarrow  wuz < wvz  \text{ for any }u, v, w,z\in \calm(Z)$$
	(	e.g the  deg-lex order, see Definition \ref{Def: deg-lex order}).
		\item [(b)] A monomial order on $\frakM(Z)$ is a well-order $\leq$ on $\frakM(Z)$ such that
		$$u< v \Rightarrow q|_u<q|_v\quad\text{for all }u,v\in\frakM(Z)\text{ and } q\in \frakM^\star(Z). $$
		\end{itemize}
	
\end{defn}

We need  another notation.
Let $Z$ be a set. For  $u\in\frakM(Z)$ with $u\neq1$, as  $u$ can be uniquely written as a product $ u_1 \cdots u_n $ with $ u_i \in Z \cup \left \lfloor\frakM(Z)\right \rfloor$ for $1\leq i\leq n$, call $n$ the breadth of $u$, denoted by $|u|$; for $u=1$, we define
		$ |u| = 0 $.
\begin{defn} [\cite{BokutChenQiu, GSZ, GGSZ, GaoGuo17}]
	Let $\leq$ be a monomial order on $\frakM(Z)$ and $ f, g \in \bk\frakM(Z) $ be monic.
	\begin{itemize}
		\item[(a)]If there are $w,u,v\in \frakM(Z)$ such that $w=\bar{f}u=v\bar{g}$ with max$\left\lbrace |\bar{f}| ,|\bar{g}|   \right\rbrace < \left| w\right| < |\bar{f}| +|\bar{g}|$, we call
		$$\left( f,g\right) ^{u,v}_w:=fu-vg$$
		the intersection composition of $f$ and $g$ with respect to $w$.
		\item[(b)] If there are $w\in\frakM(Z)$ and $q\in\frakM^\star(Z)$ such that $w=\bar{f}=q|_{\bar{g}}$, we call
		$$\left( f,g\right)^q_w:=f-q|_g $$
		the inclusion composition of $f$ and $g$ with respect to $w$.
		
	\end{itemize}
\end{defn}
\begin{defn} [\cite{BokutChenQiu, GSZ, GGSZ, GaoGuo17}]
	Let $Z$ be a set and $\leq$ a monomial order on $\frakM(Z)$. Let $\calg\subseteq \bk\frakM(Z) $.
	\begin{itemize}
		\item[(a)]An element $f\in \bk\frakM(Z) $ is called trivial modulo $(\calg,w)$ for $w\in \frakM(Z)$ if
		$$f=\underset{i}{\sum}c_iq_i|_{s_i}\text{ with }q_i|_{\bar{s_i}}<w\text{, where }c_i\in \bk,\ q_i\in\frakM^\star(Z)\  \mathrm{and}\  s_i\in \calg.$$
		\item[(b)]  The subset $\calg\subseteq \bk\frakM(Z) $ is called a  GS basis in $\bk\frakM(Z) $ with respect to $\leq$ if, for all pairs $f,g\in \calg$ monicized with respect to $\leq$, every  intersection composition of the form $\left( f,g\right) ^{u,v}_w$ is trivial modulo $(\calg,w)$,  and every inclusion composition of the form $\left( f,g\right)^q_w$ is trivial modulo $(\calg,w)$.
	\end{itemize}
\end{defn}

\textit{ {To distinguish from usual GS bases for associative algebras,  from now on, we shall rename       GS bases  in operated contexts by operated GS bases.}}


\begin{theorem} [\cite{BokutChenQiu, GSZ, GGSZ, GaoGuo17}]\label{Thm: CD}
	(Composition-Diamond Lemma) Let $Z$ be a set, $\leq$ a monomial order on $\frakM(Z)$ and $\calg\subseteq \bk\frakM(Z) $. Then the following conditions are equivalent:
	\begin{itemize}
		\item[(a)] $\calg$ is an operated GS basis in $\bk\frakM(Z) $.
		\item[(b)] Let $\eta:\bk\frakM(Z) \rightarrow\bk\frakM(Z) \slash\left\langle \calg \right\rangle_\uOpAlg$ be the quotient morphism. Denote
		$$\Irr(\calg):=\frakM(Z)\backslash \left\lbrace q|_{\overline{s}}~|~s\in \calg,\ \ q\in\frakM^\star(Z)\right\rbrace. $$
		As a $\bk$-space, $\bk\frakM(Z) =\bk\Irr(\calg)\oplus\left\langle \calg \right\rangle_\uOpAlg$ and     $\eta(\Irr(\calg))$ is a $\bk$-basis of $\bk\frakM(Z)\slash\left\langle\calg \right\rangle_\uOpAlg$.
	\end{itemize}
\end{theorem}

\begin{prop}\label{GS-basis for algebras considered as GS-basis for operated algebras}
	Let $Z$ be a set and $\leq$ a monomial order on $\frakM(Z)$. Clearly when  restricted to $ \calm(Z)$, it is still a monomial order.
	 Then a GS basis $G \subseteq \bk \calm(Z) $    with respect to the restriction of $\leq$ to $ \calm(Z)$ is also an operated GS basis in  $\bk\frakM(Z) $ with respect to $\leq$. 
\end{prop}
\begin{proof}
	For any intersection composition $\left( f,g\right) ^{u,v}_w$ in $\bk\frakM(Z) $ with $f,g\in  G \subseteq \bk \calm(Z) $, we have $w,u,v\in \cals(Z)$ and $\left( f,g\right) ^{u,v}_w$ is trivial modulo $(G,w)$ in $\bk\frakM(Z) $;  the case of inclusion compositions is similar. It follows that $G$ is a GS  basis in  $\bk\frakM(Z)$. \end{proof}

Now, let's consider operated  GS  bases for free $\Phi$-algebras over  unital algebras.   Recall that we always assume   that all OPIs are multilinear.

\begin{defn}[\cite{GaoGuo17}]
	Let $X$ be a set and $\Phi\subseteq \bk\frakM(X)$ a system of OPIs. Let $Z$ be a set and $\leq$ a monomial
	order on $\frakM(Z)$. We call $\Phi$ operated GS  on $Z$ with respect to $\leq$ if $S_{\Phi}(Z)$ is a
	GS basis  in $\bk\frakM(Z)$ with respect to $\leq$.
	We call $\Phi$ operated GS  if, for each set $Z$, there is a monomial order $\leq$ on $\frakM(Z)$ such that $\Phi$ is GS  on $Z$ with respect to $\leq$.

\end{defn}

\begin{theorem}\label{Thm: GS basis for free Phi algebra over alg}
	Let $X$ be a set and $\Phi\subseteq \bk\frakM(X)$ a system of OPIs. Let $A=\bk \calm (Z) \slash I_A$ be a unital algebra with generating set $Z$.
	Assume that $\Phi$ is operated GS  on $Z$ with respect to a monomial order $\leq$ in $\frakM(Z)$ and that  $G$ is a GS    basis of $I_{A}$ in $\bk \calm (Z)$ with respect to the restriction of $\leq$ to $ \calm(Z)$.

 Suppose that the leading monomial  of any OPI $\phi(x_1, \dots, x_n)\in \Phi$ has no subword in $\cals(X)\backslash X$,  and  that  for all $u_1, \dots, u_n\in \frakM (Z)$,  $\phi(u_1, \dots, u_n) $   vanishes or its  leading monomial is still    $\overline{\phi}(u_1, \dots, u_n) $.  Then $S_{\Phi}(Z)\cup G$ is an operated  GS basis of $\left\langle S_{\Phi}(Z)\cup I_A\right\rangle_\uOpAlg$ in $\bk\frakM(Z)$ with respect to $\leq$.

\end{theorem}

\begin{proof}
	By Proposition~\ref{Prop: from unital algebra to Phi algebra with relations},  we have
	 $ 	\calf_{\uAlg}^{\Phi\zhx\Alg}(A)=\bk\frakM(Z)\slash\left\langle S_{\Phi}(Z)\cup  I_{A}\right\rangle_\uOpAlg. $

	By Proposition~\ref{GS-basis for algebras considered as GS-basis for operated algebras},  $G$  is an operated  GS basis in  $\bk\frakM(Z) $ with respect to $\leq$ and
    by  assumption,   so is  $S_{\Phi}(Z)$. To show that $S_{\Phi}(Z)\cup G$ is an operated  GS basis of $\left\langle S_{\Phi}(Z)\cup I_A\right\rangle_\uOpAlg$ in $\bk\frakM(Z)$,    we  need to check the triviality of    any inclusion composition $\left( f,g\right)^q_{w_1}$    modulo $(S_{\Phi}(Z)\cup G,w_1)$, and also that of    any intersection composition $(f,g)^{u,v}_{w_2}$     modulo    $(S_{\Phi}(Z)\cup G,w_2)$,   for  only ($f\in S_{\Phi}(Z)$, $g\in G$) or ($f\in G$, $g\in S_{\Phi}(Z)$).


\medskip

For an inclusion composition $\left( f,g\right)^q_{w_1}$, it is easy to see that we only need to consider the case $f\in S_{\Phi}(Z)$, $g\in G$. Since for  $w_1\in\frakM(Z)$ and $q\in\frakM^\star(Z)$, we have  $w_1=\bar{f}=q|_{\bar{g}}$, we can  write $f=\phi(u_1,\dots,u_k)\in S_{\Phi}(Z)$ with $\phi(x_1,\dots,x_k)\in\Phi$ and $u_1,\dots,u_k\in \frakM(Z)$.
By assumption,  the leading monomial of any OPI  in $\Phi$ has no subword in $\cals(X)\backslash X$, so  there exists $u_i$ and $q^{\prime}\in\frakM^\star(Z)$ such that $u_i=q^{\prime}|_{\bar{g}}$. One gets $f=\sum_jc_jq_j|_{\bar{g}}$ with $q_j\in\frakM^\star(Z)$.
By multilinearity of $\phi$,
	$$\left( f,g\right)^q_{w_1}:=f-q|_g=f-f|_{\bar{g}\rightarrow g}+f|_{\bar{g}\rightarrow g}-\bar{f}|_{\bar{g}\rightarrow g}=f|_{\bar{g}\rightarrow\bar{g}-g}+(f-\bar{f})|_{\bar{g}\rightarrow g}.$$
	Observe three facts:    $f|_{\bar{g}\rightarrow\bar{g}-g}\in S_{\Phi}(Z)$, $f|_{\bar{g}\rightarrow\bar{g}-g}<w_1$ and $(f-\bar{f})|_{\bar{g}\rightarrow g}$ is trivial modulo $(G,w_1)$. Consequently, we obtain that  $\left( f,g\right)^q_{w_1}$ is trivial modulo $(S_{\Phi}(Z)\cup G,w_1)$.

\medskip

For  an intersection composition  $(f,g)^{u,v}_{w_2}$,  we only consider the case $f\in S_{\Phi}(Z)$, $g\in G$, the other case being similar.   Our method is to reduce this case  to the case of inclusion compositions.

Write  $f=\phi(u_1,\dots,u_k)\in S_{\Phi}(Z)$  with $\phi(x_1,\dots,x_k)\in\Phi$ and $u_1,\dots,u_k\in \frakM(Z)$.
For  the intersection composition  $(f,g)^{u,v}_{w_2}$,  we have $w_2=\bar{f}u=v\bar{g}$ with max$\left\lbrace |\bar{f}| ,|\bar{g}|   \right\rbrace < \left| w_2\right| < |\bar{f}| +|\bar{g}|$.   So $\bar{f}$ must have a right factor in $\cals(Z)$ and thus $\bar{\phi}$ has a right factor in $\cals(X)$. Since the leading monomial of any OPI in $\Phi$ has no subword in $\cals(X)\backslash X$, without loss of generality, we can write $\bar{\phi}(x_1,\dots,x_k)$ as $\psi(x_1,\dots,x_{k-1})x_k$, for some $\psi(x_1,\dots,x_{k-1})\in \frakM(X)$ with no right factor in $\cals(X)$. So we have $$v\bar{g}=w_2=\bar{f}u=\bar{\phi}(u_1,\dots,u_k)u=\psi (u_1,\dots,u_{k-1})u_ku=\bar{\phi}(u_1,\dots,u_ku),$$
and $u_ku=s\bar{g}$ for some $s\in \frakM(Z)$.
Let $q=\bar{\phi}(u_1,\dots,u_{k-1},s\star)\in \frakM^\star(Z)$, then $$w_2=\bar\phi(u_1,\dots,u_ku)=\bar\phi(u_1,\dots,u_{k-1},s\bar g)=q|_{\bar g}.$$
Notice that we obtain  an inclusion composition:
$$(\phi(u_1,\dots,u_ku),g)^{q}_{w_2}=\phi(u_1,\dots,u_ku)-\bar{\phi}(u_1,\dots,u_{k-1},sg)=\phi(u_1,\dots,u_ku)-vg,$$
 which  is trivial modulo $(S_{\Phi}(Z)\cup G,w_2)$ by the case of inclusion compositions  already dealt with above.

Consider two elements $\phi(u_1,\dots,u_ku)$ and $\phi(u_1,\dots,u_k)$ in $\bk\frakM(Z)$.
Let $q^\prime=\star u\in \frakM^\star(Z)$, then we have $$w_2=\bar\phi(u_1,\dots,u_ku)=\bar\phi(u_1,\dots,u_k)u=q^\prime|_{\bar\phi(u_1,\dots,u_k)}.$$
One gets  another inclusion composition: $$(\phi(u_1,\dots,u_ku),\phi(u_1,\dots,u_k))^{q^\prime}_{w_2}=\phi(u_1,\dots,u_ku)-\phi(u_1,\dots,u_k)u,$$
which is trivial modulo $(S_{\Phi}(Z),w_2)$ by the assumption that $S_{\Phi}(Z)$ is a GS  basis in $\bk\frakM(Z)$. Now it  is obvious   that $(\phi(u_1,\dots,u_ku),\phi(u_1,\dots,u_k))^{q^\prime}_{w_2}$ is also trivial modulo $(S_{\Phi}(Z)\cup G,w_2)$.

Return to our   intersection composition and we obtain that
	$$	\begin{array}{lcl}		
		\left( f,g\right) ^{u,v}_{w_2}&=&\phi(u_1,\dots,u_k)u-vg\\
		&=&\phi(u_1,\dots,u_k)u-\phi(u_1,\dots,u_ku)+\phi(u_1,\dots,u_ku)-vg\\
		&=&-(\phi(u_1,\dots,u_ku),\phi(u_1,\dots,u_k))^{q^\prime}_{w_2}+(\phi(u_1,\dots,u_ku),g)^{q}_{w_2}\\
\end{array}	$$
is trivial modulo $(S_{\Phi}(Z)\cup G,w_2)$, by the triviality of the two inclusion compositions proved before.



\end{proof}

By Theorem~\ref{Thm: CD}, we have the following result.
\begin{coro}\label{Cor: linear basis for free Phi algebra over alg}
	With the same assumption as Theorem~\ref{Thm: GS basis for free Phi algebra over alg}, denote by $$\eta:\bk\frakM(Z)\rightarrow{\calf_{\uAlg}^{\Phi\zhx\Alg}(A)=\bk\frakM(Z)\slash\left\langle { S_{\Phi}(Z)}\cup I_A\right\rangle_\uOpAlg}$$
	   the natural quotient  map.
 Then $\eta(\Irr(S_{\Phi}(Z)\cup G))$  is a $\bk$-basis of ${\calf_{\uAlg}^{\Phi\zhx\Alg}(A)}$.
\end{coro}

\begin{remark}\label{Rem: diff algebra with strange leading term}
	In Theorem~\ref{Thm: GS basis for free Phi algebra over alg}, it is necessary to require that the leading monomial  of any OPI in $\Phi$ has no subword in $\cals(X)\backslash X$.

 For example, let $A:=\bk\calm(\{z_1,z_2\})\slash (z_1z_2-1)$. It is obvious that  $\{z_1z_2-1\}$ is a GS  basis in $\bk\calm(\{z_1,z_2\})$.  Consider the OPI $$\phi(x,y)=\lfloor xy\rfloor-x\lfloor y\rfloor-\lfloor x\rfloor y,$$
	which defines differential algebras of weight zero. We know that $S_{\phi}(\{z_1,z_2\})$ is an operated  GS basis in $\bk\frakM(\{z_1,z_2\})$ with respect to a certain monomial order; see \cite{GSZ}.

However, let $w=\lfloor z_1z_2\rfloor$ and $q=\lfloor \star\rfloor\in\frakM^\star(\{z_1,z_2\})$, then $w=\bar\phi(z_1,z_2)=q|_{z_1z_2}$.
    There is an inclusion composition
	$$(\phi(z_1,z_2),z_1z_2-1)_w^q=\phi(z_1,z_2)-\lfloor z_1z_2-1\rfloor=-z_1\lfloor z_2\rfloor-\lfloor z_1\rfloor z_2,$$
	which it is not trivial modulo $S_{\phi}(\{z_1,z_2\})\cup \{z_1z_2-1\}$.
	So the set $S_{\phi}(\{z_1,z_2\})\cup \{z_1z_2-1\}$ is not an operated  GS basis in $\bk\frakM(\{z_1,z_2\})$ with respect to the given monomial order.
\end{remark}

\section{Examples of operated GS bases for free $\Phi$-algebras over Algebras}

In this section, we present many examples for which Theorem~\ref{Thm: GS basis for free Phi algebra over alg} will provide operated GS bases.
These  include  all
 OPIs  of Rota-Baxter type \cite{GGSZ, GaoGuo17}, a   class of OPIs of differential type \cite{GSZ,GaoGuo17} and  two  other examples, say,  averaging algebras and Reynolds algebras.

 \subsection{Preliminaries about rewriting systems}

 We need some basic notions  about rewriting systems in order to introduce OPIs of Rota-Baxter type and of differential type. The basic references about rewriting systems are, for instance,  \cite{BN} and the recent lecture notes \cite{Malbos}.
 \begin{defn}
Let $V$ be a $\bf{k}$-space with a $\bf{k}$-basis $Z$.
	\begin{itemize}
		\item[(a)] For $f=\sum_{w \in Z} c_{w} w \in V$ with $c_{w} \in \mathbf{k},$ the support $\operatorname{Supp}(f)$ of $f$ is the set $\{w \in Z \mid c_{w} \neq 0\} .$ By convention, we take $\operatorname{Supp}(0)=\emptyset$.
		\item[(b)]  Let $f, g \in V$. We use $f \dot{+} g$ to indicate the property that $\operatorname{Supp}(f) \cap \operatorname{Supp}(g)=\emptyset .$ If this is the case, we say $f+g$ is a direct sum of $f$ and $g,$ and use $f \dot{+} g$ also for the sum $f+g$.
		\item[(c)] For $f \in V$ and $w \in \operatorname{Supp}(f)$ with the coefficient $c_{w}$, write $R_{w}(f):=c_{w} w-f \in V$. So $f=c_{w} w+\left(-R_{w}(f)\right)$.
	\end{itemize}
\end{defn}

\begin{defn}
Let $V$ be a $\bf{k}$-space with a $\bf{k}$-basis $Z$.
	\begin{itemize}
		\item[(a)] A term-rewriting system $\Pi$ on $V$ with respect to $Z$ is a binary relation $\Pi \subseteq Z \times V$. An element $(t, v) \in \Pi$ is called a (term-)rewriting rule of $\Pi,$ denoted by $t \rightarrow v$.
		\item[(b)] The term-rewriting system $\Pi$ is called simple with respect to $Z$ if $t \dot{+} v$ for all $t \rightarrow v \in \Pi$
		\item[(c)]  If $f=c_{t} t+\left(-R_{t}(f)\right) \in V,$ using the rewriting rule $t \rightarrow v,$ we get a new element $g:=c_{t} v-R_{t}(f) \in V,$ called a one-step rewriting of $f$ and denoted by $f \rightarrow_{\Pi} g$.
		\item[(d)]  The reflexive-transitive closure of $\rightarrow_{\Pi}$ (as a binary relation on $V$ ) is denoted by $\stackrel{\ast}{\rightarrow}_{\Pi}$ and, if $f \stackrel{\ast}{\rightarrow}_\Pi g,$ we say $f$ rewrites to $g$ with respect to $\Pi$.
		\item[(e)] Two elements $f,g\in V$ are joinable if there exists $h\in V$ such that $f \stackrel{\ast}{\rightarrow}_{\Pi}h$ and $g \stackrel{\ast}{\rightarrow}_{\Pi}h$; we denote this by $f\downarrow_{\Pi} g$.
		\item[(f)] A term-rewriting system $\Pi$ on $V$ is called terminating if there is no infinite chain of one-step rewriting $$f_{0} \rightarrow_{\Pi} f_{1} \rightarrow_{\Pi} f_{2} \cdots.$$
		\item[(g)] A term-rewriting system $\Pi$ is called compatible with a linear order $\geq$ on $Z$, if $t>\bar v$ for each $t\rightarrow v\in\Pi$.
	\end{itemize}
\end{defn}

\subsection{Single  OPI of   Rota-Baxter type}\

Let us introduce OPIs of Rota-Baxter type following \cite{GGSZ,GaoGuo17}.

\begin{defn}[\cite{GGSZ,GaoGuo17}]\label{Def: OPI of RB type}
An OPI $\phi$ is said to be of Rota-Baxter type if $\phi$ is of the form $\lfloor x\rfloor\lfloor y\rfloor-\lfloor B(x, y)\rfloor,$ where $B(x, y)$ satisfies the following conditions:
	\begin{itemize}
		\item[(a)] $B(x, y)$ is linear in $x$ and $y$;
		\item[(b)] no monomial of $B(x, y)$ contains any subword of the form $\lfloor u\rfloor\lfloor v\rfloor$ for any $u,v\in\frakM(\{x,y\})\backslash\{1\}$;
		\item[(c)] for every set $Z$, the rewriting system $\Pi_{\phi}(Z):= \{\lfloor u\rfloor\lfloor v\rfloor\rightarrow\lfloor B(u,v)\rfloor~|~u,v \in \frakM(Z)\}$ is terminating;
		\item[(d)] for any set $Z$ and $u, v, w \in\frakM(Z)\backslash\{1\}$, $$B(B(u,v), w)-B(u, B(v, w))\stackrel{\ast}\rightarrow_{\Pi_{\phi}(Z)}0.$$
	\end{itemize}
\end{defn}

Gao, Guo, Sit and Zheng  gave a list of OPIs of Rota-Baxeter type and   they conjectured that these are all possible OPIs of Rota-Bax ter type.
\begin{List}[{\cite[Conjecture 2.37]{GGSZ}}]\label{Ex: list of OPIs of RB type}\
For any $c,\lambda\in\bk$, the OPI $\phi:= \lfloor x\rfloor\lfloor y\rfloor-\lfloor B(x, y)\rfloor$, where $B(x, y)$ is taken from the list below, is of Rota- Baxter type.
	\begin{itemize}
		\item[(1)] $x\lfloor y\rfloor \quad$(average operator),
		\item[(2)] $\lfloor x\rfloor y\quad$ (inverse average operator),
		\item[(3)] $x\lfloor y\rfloor+y\lfloor x\rfloor$,
		\item[(4)] $\lfloor x\rfloor y+\lfloor y\rfloor x$,
		\item[(5)] $x\lfloor y\rfloor+\lfloor x\rfloor y-\lfloor x y\rfloor \quad$ (Nijenhuis operator),
		\item[(6)] $x\lfloor y\rfloor+\lfloor x\rfloor y+ \lambda xy \quad$ (Rota-Baxter operator of weight $\lambda$),
		\item[(7)] $ x\lfloor y\rfloor-x\lfloor 1\rfloor y+\lambda x y$,
		\item[(8)] $\lfloor x\rfloor y-x\lfloor 1\rfloor y+\lambda x y$,
		\item[(9)]  $x\lfloor y\rfloor+\lfloor x\rfloor y-x\lfloor 1\rfloor y+\lambda x y \quad$ (generalized Leroux TD operator with weight $\lambda$),
		\item[(10)]  $x\lfloor y\rfloor+\lfloor x\rfloor y-x y\lfloor 1\rfloor-x\lfloor 1\rfloor y+\lambda x y$,
		\item[(11)]  $x\lfloor y\rfloor+\lfloor x\rfloor y-x\lfloor 1\rfloor y-\lfloor x y\rfloor+\lambda x y$,
		\item[(12)] $x\lfloor y\rfloor+\lfloor x\rfloor y-x\lfloor 1\rfloor y-\lfloor 1\rfloor x y+\lambda x y$,
		\item[(13)] $ c x\lfloor 1\rfloor y+\lambda x y \quad$(generalized endomorphisms),
		\item[(14)] $ c y\lfloor 1\rfloor x+\lambda y x \quad$(generalized antimorphisms).
	\end{itemize}
\end{List}

It is showed in \cite[Theorem 4.9]{GGSZ} that for an OPI $\phi$ of Rota-Baxter type, $S_{\phi}(Z)$ is a GS  basis of $\left\langle S_{\phi}(Z)\right\rangle_\uOpAlg$ with respect to some monomial order, denoted by $\leq_{\rm{db}}$ (see  \cite[Lemma 5.5]{GGSZ});   the leading monomial of $\phi$ under $\leq_{\rm{db}}$ is $\lfloor x\rfloor\lfloor y\rfloor$,  and
  the restriction of this order on $\TT(V)$ is the degree lexicographical order $\leq_{\text {dlex}}$ defined as follows.
\begin{defn}\label{Def: deg-lex order}
Let $Z$ be a set endowed with a well order $\leq_{Z}.$ By convention, define  $\operatorname{deg}_{Z}(1)=0;$ for $u=u_{1} \cdots u_{r} \in \calm(Z)\setminus \{1\}$ with $u_{1}, \dots, u_{r} \in Z,$ define $\operatorname{deg}_{Z}(u)=r.$

Define the degree lexicographical order $\leq_{\rm {dlex }}$ on $\calm(Z)$ by taking, for any $u, v \in \calm(Z)$, $u <_{\rm {dlex}}v$   if
	\begin{itemize}
		\item[(a)] either $\operatorname{deg}_{Z}(u)<\operatorname{deg}_{Z}(v)$, or
		\item[(b)] $\operatorname{deg}_{Z}(u)=\operatorname{deg}_{Z}(v)$, and $u=mu_{i}n$, $v=mv_{i}n^\prime$ for some $m,n,n^\prime\in \calm(Z)$ and $u_{i},v_{i}\in Z$ with $u_{i}<_{Z} v_{i}$.
	\end{itemize}
\end{defn}

\medskip

It is surprising to see that Theorem~\ref{Thm: GS basis for free Phi algebra over alg} could deal with all OPIs of Rota-Baxter type.
\begin{theorem}\label{Thm: GS basis for free Rota type over algebra}
	Let $Z$ be a set, $A=\bk \calm(Z)\slash I_A$ a unital $\bk$-algebra and $\phi$ an OPI of Rota-Baxter type.  Then we have:
	$$\calf^{\phi\zhx\Alg}_{\uAlg}(A)=\bk\frakM(Z)\slash\left\langle S_{\phi}(Z)\cup I_A\right\rangle_\uOpAlg.$$
	Moreover, assume $I_A$  has a GS  basis $G$ with respect to the degree lexicographical order $\leq_{\rm {dlex}}$. Then $S_{\phi}(Z)\cup G$ is an operated  GS  basis of $\left\langle S_{\phi}(Z)\cup I_A\right\rangle_\uOpAlg$ in $\bk\frakM(Z) $ with respect to $\leq_{\rm {db}}$.
\end{theorem}

\begin{proof}
Since $\lfloor x\rfloor\lfloor y\rfloor$ has no subword in $\cals(X)\backslash X$, the  result is immediately obtained by our main result  Theorem \ref{Thm: GS basis for free Phi algebra over alg}.
\end{proof}

By Corollary~\ref{Cor: linear basis for free Phi algebra over alg}, one could determine  a linear basis of the free $\phi$-algebra of Rota-Baxter type over a unital algebra.

\begin{theorem}\label{Thm: Linear basis of RB type alg on alg}
	Let $Z$ be a set, $A=\bk \calm(Z)\slash I_A$ a unital algebra with a GS  basis $G$ with respect to $\leq_{\rm{dlex}}$. Let $\phi$ be an OPI of Rota-Baxter type.  Then $$\Irr(S_{\phi}(Z)\cup G)=\frakM(Z)\backslash \left\lbrace q|_{\bar{s}},q|_{\lfloor u\rfloor\lfloor v\rfloor}~|~s\in G,q\in\frakM^\star(Z),u,v\in\frakM(Z)\right\rbrace$$ is a basis of the free $\phi$-algebra $\calf^{\phi\zhx\Alg}_{\uAlg}(A)$ over $A$.
\end{theorem}

\begin{exam}
	Consider the OPI $\phi_\Nij$ of type (5) in Example~\ref{Ex: list of OPIs of RB type}:
	$$\phi_{\Nij}(x,y)=\left\lfloor x\right\rfloor\left\lfloor y\right\rfloor-\left\lfloor\left\lfloor x\right\rfloor y\right\rfloor-\left\lfloor x\left\lfloor y\right\rfloor\right\rfloor+\left\lfloor\left\lfloor x y\right\rfloor\right\rfloor.$$
	A $\phi_{\Nij}$-algebra is called a Nijenhuis algebra.

	In \cite{LeiGuo}, given a linear basis of a $\bk$-algebra $A$, Lei and Guo   gave a $\bk$-basis of the free Nijenhuis algebra  $\calf_{\uAlg}^{\phi_\Nij\zhx\Alg}(A)$ over $A$. When the $\bk$-basis of $A$ can be deduced from a GS basis of $A$,   it can be seen that their  result is  a consequence of  Theorem~\ref{Thm: Linear basis of RB type alg on alg}.

\end{exam}


\subsection{Single OPI of differential type }\label{Case of single OPI of differential type}

\begin{defn}[\cite{GSZ,GaoGuo17}]
An OPI $\phi$ is said to be of differential type if $\phi$ is of the form $\lfloor x y\rfloor-N(x, y)$, where $N(x, y)$ satisfies the following conditions:
	\begin{itemize}
		\item[(a)] $N(x, y)$ is linear in $x$ and $y$;
		\item[(b)] no monomial of $N(x, y)$ contains any subword of the form $\lfloor uv\rfloor$ for any $u,v\in\frakM(\{x,y\})\backslash\{1\}$;
		\item[(c)] for any set $Z$ and $u, v, w \in\frakM(Z)\backslash\{1\}$, $$N(u v, w)-N(u, v w)\stackrel{\ast}\rightarrow_{\Pi_{\phi}(Z)}0,$$
		where $\Pi_{\phi}(Z):= \{\lfloor uv\rfloor\rightarrow N(u,v)~|~u,v \in \frakM(Z)\}$.
	\end{itemize}
\end{defn}

  Guo, Sit and Zhang  gave a list of OPIs of differential type and   they conjectured that these are all possible OPIs of differential type.

\begin{List}[{\cite[Conjecture 4.7]{GSZ}}]\label{Ex: list of OPIs of diff  type}
For any $a,b,c,\lambda_{ij}\in\bk,~i, j \geq0$, the OPI $\phi:= \lfloor x y\rfloor-N(x, y)$, where $N(x, y)$ is taken from the list below, is of differential type.

	\begin{itemize}
		\item[(1)] $a(x\lfloor y\rfloor+\lfloor x\rfloor y)+b\lfloor x\rfloor\lfloor y\rfloor+cxy$ where $a^{2}=a+bc$,
		\item[(2)] $a b^{2} y x+b x y+a\lfloor y\rfloor\lfloor x\rfloor-ab(y\lfloor x\rfloor+\lfloor y\rfloor x)$,
		\item[(3)] $\sum_{i, j \geq 0} \lambda_{i j}\lfloor 1\rfloor^{i} x y\lfloor 1\rfloor^{j}$ with the convention that $\lfloor 1\rfloor^{0}=1$,
		\item[(4)] $x\lfloor y\rfloor+\lfloor x\rfloor y+a x\lfloor 1\rfloor y+b x y$,
		\item[(5)] $\lfloor x\rfloor y+a(x\lfloor 1\rfloor y-x y\lfloor 1\rfloor)$,
		\item[(6)] $x\lfloor y\rfloor+a(x\lfloor 1\rfloor y-\lfloor 1\rfloor x y)$.
	\end{itemize}
\end{List}

It is showed in \cite[Theorem 5.7]{GSZ} that for an OPI $\phi$ of differential type, $S_{\phi}(Z)$ is an operated GS  basis of $\left\langle S_{\phi}(Z)\right\rangle_\uOpAlg$ with respect to some monomial order, under which the leading monomial of $\phi$ is $\lfloor xy\rfloor$. Moreover, the restriction of this order on $\TT(V)$ is $\leq_{\rm{dlex}}$.

Since the leading term $\lfloor u v\rfloor$ contains a subword $uv\in \cals(X)\backslash X$, Theorem \ref{Thm: GS basis for free Phi algebra over alg} can not be applied to differential type directly. However, in some special cases, we can still compute the operated GS basis by rewriting the OPI. For case (1) with $a=b=0$ Example~\ref{Ex: list of OPIs of diff  type}, we substitute $xy$ by $z$ and get the following two types of  OPIs both with leading monomial $\lfloor z\rfloor$,

	\begin{itemize}
		\item[(1$^\prime$)] $\lfloor z\rfloor-cz$ for any $c$,
	\end{itemize}
They define the same $\phi$-algebras as case (1) (with $a=b=0$) in Example~\ref{Ex: list of OPIs of diff  type}. By Theorem~\ref{Thm: GS basis for free Phi algebra over alg} and Corollary~\ref{Cor: linear basis for free Phi algebra over alg}, we have the following result.

\begin{prop}\label{Prop: GS basis for differential type 1+3}
	Let $Z$ be a set, $A=\bk \calm(Z)\slash I_A$ a unital algebra with a GS  basis $G$ with respect to $\leq_{\rm{dlex}}$. Let $\phi$ be an OPI of type (1$^\prime$).  Then $S_{\phi}(Z)\cup G$ is an operated  GS  basis of $\left\langle S_{\phi}(Z)\cup I_A\right\rangle_\uOpAlg$ in $\bk\frakM(Z) $ with respect to the monomial order on $\frakM(Z)$ given in \cite{GSZ}, and
	$$\Irr(S_{\phi}(Z)\cup G)=\frakM(Z)\backslash \left\lbrace q|_{\bar{s}},q|_{\lfloor u\rfloor}~|~s\in G,q\in\frakM^\star(Z),u\in\frakM(Z)\right\rbrace=\calm(Z)\backslash \left\lbrace q|_{\bar{s}}~|~s\in G\right\rbrace$$ is a linear basis of the free $\phi$-algbra $\calf^{\phi\zhx\Alg}_{\uAlg}(A)$ over $A$.
\end{prop}

\begin{remark}
	Consider the OPI $\phi_\Dif\in\bk\frakM(\{x,y\})$ of type (1) in Example~\ref{Ex: list of OPIs of diff  type}, by taking $a=1$, $b=\lambda$ and $c=0$:
	$$\phi_\Dif(x,y)=\left\lfloor xy\right\rfloor-x\left\lfloor y\right\rfloor-\left\lfloor x\right\rfloor y-\lambda\left\lfloor x\right\rfloor\left\lfloor y\right\rfloor.$$
	A $\phi_\Dif$-algebra is called a differential algebra of weight $\lambda$.

	In \cite{GuoLi}, Guo and Li showed that  the free differential algebra $\calf_{\uAlg}^{\phi_\Dif\zhx\Alg}(A)$ over a unital $\bk$-algebra $A=\bk\calm(X)\slash I_A$ is the same as  the free differential algebra over the set $X$ modulo the differential ideal generated by the ideal $I_A$, which can be deduced from  Proposition~\ref{Prop: from unital algebra to Phi algebra with relations} as well.

Guo and Li  also gave a $\bk$-basis of $\calf_{\uAlg}^{\phi_\Dif\zhx\Alg}(A)$ by using the theory of  differential GS bases. Their method is completely different from ours.

Moreover, our result Theorem~\ref{Thm: GS basis for free Phi algebra over alg} could not apply directly to differential algebras under the monomial order introduced in \cite{GSZ}; see Remark~\ref{Rem: diff algebra with strange leading term}. However, it may be possible to modify the monomial order so that we can apply Theorem~\ref{Thm: GS basis for free Phi algebra over alg} and this task will be taken in a subsequent paper.
	
\end{remark}



\subsection{Multiple OPIs}

\begin{exam}[Averaging algebra]\
Take $\Phi_{\Av}$ to be the following set of operated polynomial identities:
	$$\{\left \lfloor x_1\right\rfloor \left\lfloor x_2\right\rfloor-\left\lfloor x_1\left \lfloor x_2 \right\rfloor \right\rfloor,~ \left\lfloor x_1\right\rfloor \left\lfloor x_2 \right\rfloor-\left\lfloor \left\lfloor x_1\right\rfloor x_2\right\rfloor\}.$$ Then a $\Phi_{Av}$-algebra is called an averaging algebra.

 Let $$\bar\Phi_{\Av}=\{\left\lfloor \left \lfloor x_1 \right\rfloor x_2\right\rfloor-\left \lfloor x_1\right\rfloor \left\lfloor x_2\right\rfloor,~ \left\lfloor x_1\left \lfloor x_2 \right\rfloor \right\rfloor-\left\lfloor x_1\right\rfloor \left\lfloor x_2 \right\rfloor,~\lfloor x_1\rfloor^{(2)}\lfloor x_2\rfloor-\lfloor x_1\rfloor\lfloor x_2\rfloor^{(2)}\}.$$ Note that   $\bar\Phi_{\Av}$  has three elements.
Consider the free averaging algebra $\calf^{\Phi_{\Av}\zhx\Alg}_{\Set}(Z)$ over a given set $Z$.
It is showed in \cite[Theorem 3.10]{GaoZhang} that the set $S_{\bar\Phi_{\Av}}(Z)$ is an operated  GS  basis of $\calf^{\Phi_{\Av}\zhx\Alg}_{\Set}(Z)$ with respect to the order defined in \cite{GSZ}
and the leading monomials of the elements in $\bar\Phi_{\Av}$ are $\lfloor\lfloor x_1\rfloor x_2\rfloor$, $\lfloor x_1\lfloor x_2\rfloor\rfloor$ and $\lfloor x_1\rfloor^{(2)}\lfloor x_2\rfloor$ respectively.

By Theorem~\ref{Thm: GS basis for free Phi algebra over alg} and Corollary~\ref{Cor: linear basis for free Phi algebra over alg}, given a unital algebra $A$ with a GS  basis $G$ with respect to $\leq_{\rm{dlex}}$, the set
$ S_{\bar\Phi_{\Av}}(Z)\cup G)$   is an operated GS basis of the free averaging algebra $\calf^{\Phi_{\Av}\zhx\Alg}_{\uAlg}(A)$ over $A$  and
the set
$$\Irr(S_{\bar\Phi_{Av}}(Z)\cup G)=\frakM(Z)\backslash \left\lbrace q|_{\bar{s}},q|_{\lfloor\lfloor u\rfloor v\rfloor},q|_{\lfloor u\lfloor v\rfloor\rfloor},q|_{\lfloor u\rfloor^{(2)}\lfloor v\rfloor}~|~s\in G,q\in\frakM^\star(Z),u,v\in\frakM(Z)\right\rbrace$$
 is a linear basis of   $\calf^{\Phi_{\Av}\zhx\Alg}_{\uAlg}(A)$.

\end{exam}

\begin{exam}[Reynolds algebra]\
Let $\phi_{\Rey}$ be the following operated polynomial identity:
\[\left\lfloor x_1\right\rfloor \left\lfloor x_2\right\rfloor-\left\lfloor x_1\left\lfloor x_2\right\rfloor \right\rfloor-\left\lfloor\left\lfloor x_1\right\rfloor x_2\right\rfloor +\left\lfloor\left\lfloor x_1\right\rfloor\left\lfloor x_2\right\rfloor\right\rfloor.\]
Then a $\phi_{\Rey}$-algebra is called a Reynolds algebra.

Consider the free Reynolds algebra $\calf^{\phi_{\Rey}\zhx\Alg}_{\Set}(Z)$ over a given set $Z$.
By Theorem~\ref{Thm: CD} and    \cite[Proposition 2.13 and Theorem 3.9]{ZhangGaoGuo}, the set $S_{\bar\Phi_{\Rey}}(Z)$ is an operated GS  basis in $\calf^{\phi_{\Rey}\zhx\Alg}_{\Set}(Z)$ with respect to the order defined in \cite{GSZ}, where
$$\bar\Phi_{\Rey}=\{\lfloor\lfloor x_1\rfloor \cdots\lfloor x_n\rfloor\rfloor-\sum_{i=1}^{n}\lfloor\lfloor x_1\rfloor\cdots\lfloor x_{i-1}\rfloor x_i \lfloor x_{i+1}\rfloor \cdots\lfloor x_n\rfloor\rfloor+\lfloor x_1\rfloor\cdots\lfloor x_n\rfloor~|~n\geq2\},$$
and the leading monomials of the elements in $\bar\Phi_{\Rey}$ are of the form $\lfloor\lfloor x_1\rfloor\lfloor x_2\rfloor \cdots\lfloor x_n\rfloor\rfloor$.

By Theorem~\ref{Thm: GS basis for free Phi algebra over alg} and Corollary~\ref{Cor: linear basis for free Phi algebra over alg}, given a unital algebra $A$ with a GS  basis $G$ with respect to $\leq_{\rm{dlex}}$, the set
$ S_{\bar\Phi_{\Rey}}(Z)\cup G$   is an operated GS basis of the free Reynolds algebra $\calf^{\phi_{\Rey}\zhx\Alg}_{\uAlg}(A)$ over $A$ and  the set
$$\Irr(S_{\bar\Phi_{\Rey}}(Z)\cup G)=\frakM(Z)\backslash \left\lbrace q|_{\bar{s}},q|_{\lfloor\lfloor u_1\rfloor \cdots\lfloor u_n\rfloor\rfloor}~|~s\in G,q\in\frakM^\star(Z),u_1,\cdots,u_n\in\frakM(Z),n\geq2\right\rbrace$$
 is a linear basis of   $\calf^{\phi_{\Rey}\zhx\Alg}_{\uAlg}(A)$.

Note that in this case, $\bar\Phi_{\Rey}$  has infinitely many elements.

\end{exam}

\end{document}